\documentclass{birkau}

\usepackage{amsmath,amssymb,tikz,tikz-cd,amsmath,mathrsfs,mathtools,multicol,dirtytalk,tabularx,xy,lipsum,url,enumitem,cmll}

\usepackage[colorlinks=true,linkcolor=blue,citecolor=blue]{hyperref}

\numberwithin{equation}{section}
\theoremstyle{plain}
\newtheorem{theorem}{Theorem}[subsection]
\newtheorem{lemma}[theorem]{Lemma}
\newtheorem{proposition}[theorem]{Proposition}
\newtheorem{corollary}[theorem]{Corollary}

\theoremstyle{definition}
\newtheorem{definition}[theorem]{Definition}
\newtheorem{notation}[theorem]{Notation}
\newtheorem{remark}[theorem]{Remark}
\newtheorem{example}[theorem]{Example}

\DeclareMathOperator{\dom}{dom}

\DeclareMathOperator{\ran}{ran}

\newcommand{\inv}[2]{#1^{-1}(#2)}
\newcommand{\0}{\emptyset}
\newcommand{\powerfin}{\mathcal{P}_{\mathrm{fin}}}
\newcommand{\xpm}{X^\pm}
\newcommand{\zenbu}[1]{\mathop{\forall#1}}
\newcommand{\aru}[1]{\mathop{\exists#1}}

\begin{document}

\title[Choice-free duality for orthocomplemented lattices]{Choice-free duality for orthocomplemented lattices by means of spectral spaces}

\author[Joseph McDonald]{Joseph McDonald}
\address{Institute for Logic, Language, and Computation\\
University of Amsterdam\\North Holland 1098 XG\\Netherlands}
\email{jsmcdon1@ualberta.ca}
\corrauthor[Kentar\^o Yamamoto]{Kentar\^o Yamamoto}
\address{Group in Logic and Methodology of Science\\
University of California at Berkeley\\California 94720-3840\\U.S.A.}
\email{ykentaro@math.berkeley.edu}

\thanks{The first researcher thanks the Institute for Logic, Language, and Computation at the University of Amsterdam for its support during the preparation of this paper. The second researcher thanks the Group in Logic and Methodology of Science at the University of California, Berkeley, the Institute of Informatics of the Czech Academy of Sciences, and the Takenaka Scholarship Foundation for their support during the preparation of this paper.}



\begin{abstract}
The existing topological representation of an orthocomplemented lattice via the clopen orthoregular subsets of a Stone space depends upon Alexander's Subbase Theorem, which asserts that a topological space $X$ is compact if every subbasic open cover of $X$ admits of a finite subcover. This is an easy consequence of the Ultrafilter Theorem---whose proof depends upon Zorn's Lemma, which is well known to be equivalent to the Axiom of Choice.  Within this work, we give a choice-free topological representation of orthocomplemented lattices by means of a special subclass of spectral spaces; choice-free in the sense that our representation avoids use of Alexander's Subbase Theorem, along with its associated nonconstructive choice principles. We then introduce a new subclass of spectral spaces which we call \emph{upper Vietoris orthospaces} in order to characterize up to homeomorphism (and isomorphism with respect to their orthospace reducts) the spectral spaces of proper lattice filters used in our representation. It is then shown how our constructions give rise to a choice-free dual equivalence of categories between the category of orthocomplemented lattices and the dual category of upper Vietoris orthospaces. Our duality combines Bezhanishvili and Holliday's choice-free spectral space approach to Stone duality for Boolean algebras with Goldblatt and Bimb\'o's choice-dependent orthospace approach to Stone duality for orthocomplemented lattices. 
\end{abstract}
\maketitle

\section{Introduction}\label{sec:intro}

Assuming Alexander's Subbase Theorem---which asserts that a topological space $X$ is compact  if every subbasic open cover of $X$ admits of a  finite subcover---Goldblatt~ \cite{goldblatt1} constructed, for an arbitrary orthocomplemented lattice $L$, a binary relational structure $X^{\pm}_L$, or an \emph{orthospace}, consisting of all proper lattice filters $\mathfrak{F}(L)$ of $L$ (with its associated patch topology, which is Stone) equipped with a binary orthogonal relation $\perp_L\subseteq\mathfrak{F}(L)\times\mathfrak{F}(L)$ which is irreflexive and symmetric. In addition, Goldblatt proved that (up to isomorphism) every orthocomplemented lattice $L$ arises via the clopen orthoregular (see Definition~\ref{orthoregularity}) subsets of $X^{\pm}_L=(X^{\pm}_L,\perp_L)$ ordered by set-theoretic inclusion. Much later, Bimb\'o in \cite{bimbo} introduced a class of topological orthospaces as a means
to characterize (up to homeomorphism and isomorphism with respect to $\perp$) the
dual space of $X^{\pm}_L$ and used this to prove that the category of
orthocomplemented lattices is dually equivalent to the category of orthospaces.

Note that the topological representation just described depends on the Axiom of Choice, as the proof of Alexander's Subbase Theorem assumes the Ultrafilter Theorem, whose proof depends upon Zorn's Lemma, which is equivalent to the Axiom of Choice. We refer to \cite{herrlich,rubin} for an in-depth exposition concerning how the above choice-principles hang together. The indispensability of the Axiom of Choice within Goldblatt's representation is a common facet among related topological representation theorems of various classes of ordered algebraic structures. Indeed, Stone's representation of Boolean algebras via Stone spaces in \cite{stone1}, Priestley's representation of distributive lattices via Priestley spaces in \cite{priestly}, Esakia's representation of Heyting algebras via Esakia spaces in \cite{esakia,esakia1}, and J\'onsson and Tarski's representation of modal algebras via modal spaces in \cite{jonsson}, all depend upon some nonconstructive choice principle.  

It was however recently demonstrated by Bezhanishvili and Holliday in \cite{bezhanishvili} that a choice-free topological representation of Boolean algebras is achievable, one which is independent of the Boolean Prime Ideal Theorem. Whereas Stone's choice-dependent representation for Boolean algebras shows that any Boolean algebra $B$ be can represented via the clopen sets of a Stone space $X$, Bezhanishvili and Holliday demonstrated independently of the Boolean Prime Ideal Theorem that every Boolean algebra $B$ arises via the compact open subsets of a spectral space $X$, which are also regular open in the Alexandroff topology $\mathcal{UP}(X,\leqslant)$ where $\leqslant$ is the specialization order over $X$. In addition, they established a choice-free categorical dual equivalence between the category of Boolean algebras and Boolean homomomorphisms and a the dual category of upper Vietoris spaces and spectral $p$-morphisms. 

Their techniques stemmed from Stone's observation in \cite{stone2} that distributive lattices can be represented via the compact open subsets of a subclass of spectral spaces as well as Tarski's discovery in \cite{tarski1,tarski2} that the regular open subsets of a spectral space give rise to a Boolean algebra. In addition, they incorporated techniques developed by Vietoris in \cite{vietoris} as the subclass of spectral spaces they employ can also be shown as arising as the hyperspace of closed non-empty subsets of a Stone space that comes equipped with the upper Vietoris topology. The duality established in \cite{bezhanishvili} is closely related to Jipsen and Moshier's duality for arbitrary lattices developed in \cite{moshier} in that they both use spaces of all (proper) filters.
More general duality results include a work by Hofman, Mislove, and Stralka~\cite{hms} for semilattices and another by Gonz\'alez and Jansana~\cite{gj2016} for posets.

Within this work, we combine Bezhanishvili and Holliday's choice-free spectral space approach to Stone duality for Boolean algebras with Goldblatt and Bimb\'o's choice-dependent orthospace approach to Stone duality for orthocomplemented lattices as a means to prove a choice-free topological representation theorem for the class of orthocomplemented lattices by means of a special subclass of spectral spaces, independently of Alexander's Subbase Theorem and its associated nonconstructive choice principles. We then introduce a new subclass of spectral spaces which we call \emph{upper Vietoris orthospaces} as a means to characterize (up to homeomorphism and isomorphism with respect to $\perp$) the spectral space of proper lattice filters used in our representation. We then prove that the category induced by this class of spectral spaces, along with their associated weak $p$-morphisms, is dually equivalent to the category of orthocomplemented lattices, along with their associated lattice homomorphisms. In light of this duality, we proceed by developing a \say{duality dictionary} which establishes how various lattice-theoretic concepts (as applied to orthocomplemented lattices) can be translated into their corresponding dual upper Vietoris orthospace counterparts.

Throughout the present paper, we assume the general motivations discussed by Herrlich in \cite{herrlich} of investigating mathematical structures based on ZF instead of ZFC and also assume the motivations in \cite{bezhanishvili} of applying this general constructive (or choice-free) approach to mathematics to the topological duality theory of ordered algebraic structures.

Our motivations for studying orthocomplemented lattices is two-fold: First, orthocomplemented lattices, in comparison to Boolean algebras, Heyting algebras, distributive lattices, etc., are a relatively understudied class of lattice structures within duality theory. Second, the class of all orthocomplemented lattices contains various subclasses of lattice structures that behave as algebraic models for various quantum logics. For instance, the algebraic model for quantum logics of a finite dimensional Hilbert space is a modular lattice and the algebraic model for quantum logics of an infinite dimensional Hilbert space is an orthomodular lattice, both of which are the most important subclasses of the class of orthocomplemented lattices. These insights arose, in part, from the discoveries of Birkhoff and Von Neumann in \cite{birkhoff}.

The contents of this paper are organized in the following manner: In the
second section, we establish the basic algebra of orthocomplemented lattices
and discuss some important examples. In the third section, we investigate
orthospaces, spectral spaces, and give the promised choice-free topological
representation theorem for orthcomplemented lattices. In the fourth section, we characterize the choice-free duals of the spectral spaces used in our representation. In the fifth section, we prove the promised choice-free categorical
dual equivalence theorem. In light of our duality theorem, in the sixth section we develop a \say{duality dictionary} which
establishes how various lattice-theoretic concepts (as applied to orthocomplemented lattices) can be translated into their corresponding dual UVO-space counterparts.

\section{Orthocomplemented lattices}
In this section, we review the basics of the theory of
orthocomplemented lattices. For a more detailed treatment of orthocomplemented lattices and important subclasses of these lattices, refer to MacLaren in \cite{maclaren}, Bruns and Harding in \cite{bruns}, and Kalmbach in \cite{kalmbach}.  
\subsection{Foundations}

 We begin by defining the class of orthocomplemented lattices as a variety (presentable in possibly many distinct signatures) characterized by satisfying finitely many equations. 
\begin{definition}\label{lattice1} If $L=(L;\wedge,\vee,^{\perp},0,1)$ is an algebra of type $(2,2,1,0,0)$, then $L$ is an \textit{orthocomplemented lattice} (henceforth, an \emph{ortholattice}) when the following equations are satisfied:
\setlength{\columnsep}{0in}
  \begin{multicols}{2}
\begin{enumerate}
    \item $a\wedge (b\wedge c)=(a\wedge b)\wedge c$\label{assoc1}
    \item $a\vee( b\vee c)=(a\vee b)\vee c$\label{assoc2}
    \item $a\wedge b=b\wedge c$\label{comm1}
    \item $a\vee b=b\vee c$\label{comm2}
    \item $a\wedge(b\vee a)=a$\label{abs1}
    \item $a\vee(b\wedge a)=a$\label{abs2}
  
    \item $1\wedge a=a$\label{unit1}
    \item $0\vee a=a$\label{unit2}
    \item $(a\wedge b)^{\perp}=a^{\perp}\vee b^{\perp}$\label{dem1}
    \item $(a\vee b)^{\perp}=a^{\perp}\wedge b^{\perp}$\label{dem2}
    \item $(a^{\perp}\wedge b^{\perp})^{\perp}=a\vee b$\label{dem3}
    \item $(a^{\perp}\vee b^{\perp})^{\perp}=a\wedge b$\label{dem4}
    \item $a\wedge a^{\perp}=0$\label{comp1}
    \item $a\vee a^{\perp}=1$.\label{comp2}
\end{enumerate}
\end{multicols}
\end{definition}
Observe that the above formulation guarantees that every ortholattice is a bounded complemented lattice satisfying De Morgan's distribution laws for orthocomplements over meets and joins, so that they are interdefinable lattice operations with respect to orthocomplements.   
\begin{definition}\label{lattice2}
 If $L=(L;\wedge,^{\perp},0)$ is an algebra of type $(2,1,0)$ with $a\vee b:=(a^{\perp}\wedge b^{\perp})^{\perp}$ and $1:=0^{\perp}$, then $L$ is an \textit{ortholattice} if $(L;\wedge,\vee)$ is a lattice and the following conditions are satisfied: 
 \begin{enumerate}
     \item $a\wedge a^{\perp}=0$\label{complement}
     \item $a\leq b\Longrightarrow b^{\perp}\leq a^{\perp}$\label{order reversing}
     \item $a^{\perp\perp}= a$.\label{double complement}
 \end{enumerate}
 \end{definition}
 Conditions \ref{lattice2}.\ref{order reversing} and \ref{lattice2}.\ref{double complement} guarantee that the orthocomplement operator $^{\perp}$ is  a dual order isomorphism that is an involution. That the above two formulations of an ortholattice coincide can be easily verified. 
 
  Although the equations within Definition \ref{lattice1} include some redundancies, they make explicit the fact that the class of ortholattices can simply be viewed as a variety in which the join and the meet operations need not satisfy the distributive law, a property characteristic of Boolean algebras. In fact, an algebra $B=(A;\wedge,\vee,
  {}^{\perp},0,1)$ of type $(2,2,1,0,0)$ is a Boolean algebra when $B$ satisfies the equations within Definition \ref{lattice1} and in addition, satisfies the following distribution laws: 
  \begin{equation}\label{dist1}
     a\wedge(b\vee c)=(a\wedge b)\vee(a\wedge c),\hspace{.1cm}
      a\vee(b\wedge c)=(a\vee b)\wedge(a\vee c).
  \end{equation}
Given that Definitions \ref{lattice1} and \ref{lattice2} of an ortholattice are equivalent, we will adopt the latter for the sake of simplicity. The Hasse diagrams depicted in Figure \ref{2 times 2} are examples of ortholattices.   
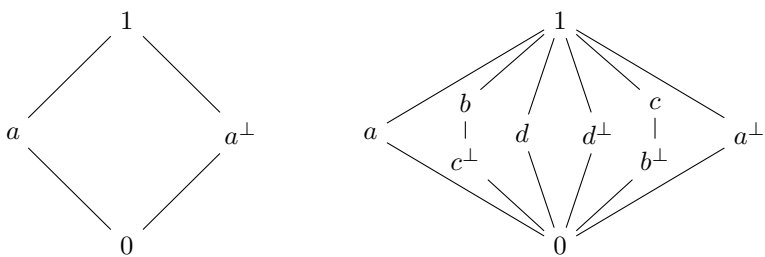
\begin{figure}[htbp]
\begin{tikzpicture}
  \node (a) at (0,2.5) {$1$};
  \node (b) at (1.5,1) {$a^{\perp}$};
  \node (c) at (-1.5,1) {$a$};
  \node (d) at (0,-0.5) {$0$};
  \draw (a) -- (b) (b) -- (d) (d) -- (c) (c) -- (a) ;
  \draw[preaction={draw=white, -,line width=6pt}];
\end{tikzpicture}
\hskip 3em 
\begin{tikzpicture}
  \node (a) at (0,2.5) {$1$};
  \node (b) at (2.5,1) {$a^{\perp}$};
  \node (c) at (-2.5,1) {$a$};
  \node (d) at (0,-0.5) {$0$};
  \node (e) at (-1.25,.65) {$c^{\perp}$};
  \node (f) at (-1.25,1.4) {$b$};
    \node (g) at (1.25,.65) {$b^{\perp}$};
  \node (h) at (1.25,1.4) {$c$};
  \node (i) at (.5,1) {$d^{\perp}$};
  \node (j) at (-.5,1) {$d$};
  \draw (a) -- (b) (b) -- (d) (d) -- (c) (c) -- (a) (e) -- (f) (g) -- (h) (e) -- (d) (g) -- (d) (f) -- (a) (h) -- (a) (i) -- (a) (j) -- (a) (i) -- (d) (j) -- (d);
  \draw[preaction={draw=white, -,line width=6pt}];
\end{tikzpicture}
\caption{the lattices $2\times 2$ and $O_{10}$}\label{2 times 2}
\end{figure}

Clearly, the $2\times 2$ lattice is an example of an ortholattice which is also a Boolean algebra and hence a distributive lattice. The fact that ortholattices however in general drop the distributive property is easily exhibited within the $O_{10}$ ortholattice which admits of sublattices $A,B\subseteq O_{10}$ that are isomorphic to the $M_3$ and $N_5$ lattices, depicted in Figure \ref{m3 and n5}.

 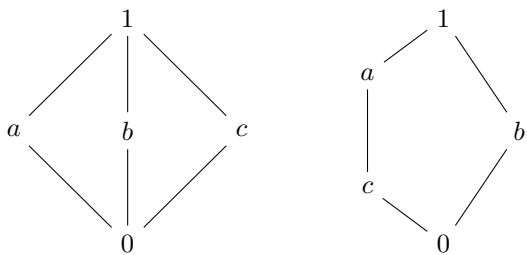
\begin{figure}[htbp]   
\begin{tikzpicture}
  \node (a) at (0,2.5) {$1$};
  \node (b) at (1.5,1) {$c$};
  \node (c) at (-1.5,1) {$a$};
  \node (d) at (0,-0.5) {$0$};
  \node (e) at (0,1) {$b$};
  \draw (a) -- (b) (b) -- (d) (d) -- (c) (c) -- (a) (a) -- (e) (d) -- (e);
  \draw[preaction={draw=white, -,line width=6pt}];
\end{tikzpicture}    
\hskip 3.25em
  \begin{tikzpicture}
  \node (a) at (0,2.5) {$1$};
     
   \node (c) at (-1,1.75) {$a$};
 \node (d) at (-1,.25) {$c$};
  \node (e) at (0,-.5) {$0$};
    \node (f) at (1,1) {$b$};
  \draw (a) -- (c) (a) -- (f) (e) -- (d) (e) -- (f) (d) -- (c) ;
  \draw[preaction={draw=white, -,line width=6pt}];
\end{tikzpicture}
\caption{The lattices $M_3$ and $N_5$}\label{m3 and n5}
\end{figure}

Note that this implies that the ortholattice $O_{10}$ is non-distributive, and hence, not a Boolean algebra. This is a consequence of the following well known characterization theorem of distributive lattices.
\begin{theorem}[Birkhoff \cite{birkhoff1} and Dedekind \cite{dedekind}]\label{birkhoff}
A lattice $L$ is distributive if and only if there exists no sublattice $A\subseteq L$ isomorphic to either $M_3$ or $N_5$.
\end{theorem}  

If $L$ and $L'$ are ortholattices, then $h\colon L\to L'$ is an \emph{ortholattice homomorphism} if $h$ preserves the ortholattice operations from $L$ to $L'$. An ortholattice homomorphism $h\colon L\to L'$ is an \emph{isomorphism} if $h$ is bijective.

\subsection{Examples}
 \begin{example}\label{example1} Every Boolean algebra $B$ with Boolean complements taken to be orthocomplements is an ortholattice. 
 \end{example}

 \begin{example}\label{example2} Let $\mathcal{H}$ be a Hilbert space over a field $F$ (such as $\mathbb{R}$ or $\mathbb{C}$); namely a real or complex valued inner product space which is also a complete metric space with respect to the metric induced by the inner product $\langle\cdot,\cdot\rangle\colon\mathcal{H}\times\mathcal{H}\to F$ associated with $\mathcal{H}$. The collection $L(\mathcal{H})$ of closed linear subspaces of $\mathcal{H}$ ordered by subspace inclusion gives rise to an ortholattice in which each closed linear subspace $X\subseteq\mathcal{H}$ admits of an orthogonal complement defined by $X^{\perp}=\{x\in\mathcal{H}\mid \forall y\in X:\langle x,y\rangle=0\}$. 
 \end{example}
 \begin{remark}
  Note that in particular, if $\mathcal{H}$ is a finite dimensional Hilbert space, then $L(\mathcal{H})$ is a modular lattice and if $\mathcal{H}$ is an infinite dimensional Hilbert space, then $L(\mathcal{H})$ is an orthomodular lattice. 
 \end{remark}
 Refer to \cite{bell} and \cite{birkhoff} for more details pertaining to the modular and orthomodular lattices induced by the lattice of closed linear subspaces of $\mathcal{H}$.

\section{Representation of ortholattices via spectral spaces}
\noindent We proceed by examining orthospaces and spectral spaces. We then demonstrate how certain spectral spaces give rise to the promised choice-free representation of ortholattices. Refer to Bell in
\cite{bell} for an in-depth exposition of the general theory of orthospaces and Dickmann, Tressl, and Schwartz in \cite{dickmann} for an in-depth exposition of the general theory of spectral spaces. 
\subsection{Orthospaces and orthoregularity}
\begin{definition}\label{orthospace}
An \textit{orthospace} is pair $(X,\perp)$ such that $X$ is a set and $\mathord\perp\subseteq X^2$ is a binary \emph{orthogonality relation} which is irreflexive (i.e., $\forall x\in X$, $x\not\perp x$) and symmetric (i.e., $\forall x,y\in X$, if $x\perp y$, then $y\perp x$). 

In this case,
\begin{enumerate}
    \item For every $x\in X$ and $Y\subseteq X$, we define $x\perp Y\Longleftrightarrow x\perp y$ for all $y\in Y$.
    \item Given any $Y\subseteq X$, we define  $Y^*=\{x\mid x\perp Y\}$.
\end{enumerate}
\end{definition}

 Informally, $Y^*$ can be thought of as is the set of all
points in $X$ that are orthogonal to every point in $Y$. The first example of an orthogonality relation we consider arises via the dot product over a vector space. 
\begin{example}\label{example3}
Let $\mathbb{R}^n$ be the  $n$-dimensional Euclidean space. Given non-zero vectors $x=[x_1,\dots,x_n]$, $y=[y_1,\dots,y_n]\in\mathbb{R}^n$, we have $x\perp y$ if and only if \[x\cdot y=\sum_{i=1}^nx_iy_i=0.\] 
\end{example}
Orthogonality relations also arise from the function space of integrable functions that form a vector space equipped with some inner product. 
\begin{example}\label{example4}
Define a continuous weight function $w$ over some real closed interval $[a,b]$. Then, two continuous functions $f,g\colon\mathbb{R}\to\mathbb{R}$ are orthogonal if \[\langle f,g\rangle_w=\int_a^bf(x)g(x)w(x)dx=0.\]  For instance, the functions $f(x)=1$ and $g(x)=x$ are orthogonal if 
\[\langle f,g\rangle_w=\int_{-1}^1f(x)g(x)dx.\]
\end{example}
 
\begin{definition}\label{orthoregularity}
Let $(X,\perp)$ be an orthospace. A subset $Y\subseteq X$ is \emph{orthoregular} (or $\perp$-\textit{regular}) if and only if $Y=Y^{**}=\{z\mid z\perp Y^*\}.$ 
\end{definition}

\begin{example}
Any closed linear subspace $X\subseteq\mathbb{R}^n$ is orthoregular in the sense that $X^{\perp\perp}=X$ since $\mathbb{R}^n=X\oplus X^{\perp}$ meaning that any vector $x=[x_1,\dots,x_n]\in\mathbb{R}^n$ can be uniquely written as $x=y+z$ with $y=[y_1,\dots,y_n]\in X$ and $z=[z_1,\dots,z_n]\in X^{\perp}$, which implies that $0=x\cdot z=(y+z)\cdot z=y\cdot z+z\cdot z=z\cdot z$ and thus $z=0$ and $x=y$.  
\end{example}

\subsection{Spectral spaces}

 It will be useful to fix the following notation for important subsets of relational topological spaces that will be studied throughout this work. 
\begin{notation}
Given a topologized orthospace $(X,\leqslant,\perp,\mathcal{T})$ where $\mathcal{T}\subseteq\mathcal{P}(X)$ is some topology and $\leqslant$ is the specialization order over $X$, we define the following collections of subsets of $X$ as follows: 
\begin{enumerate}
    \item $\mathcal{C}(X)$ is the collection of sets that are compact in $X$;
    \item $\mathcal{O}(X)$ is the collection of sets that are open in $X$;
    \item $\mathcal{R}(X)$ is the collection of sets that are orthoregular in $X$;
    \item $\mathcal{UP}(X)$ is the collection of sets that are open in the upset topology (i.e., the upwards closed or upper set topology) on $X$; 
    \item $\mathrm{RO}(X)$ is the collection of subsets that are regular open in the upset topology; $\mathcal{UP}(X, \leqslant)$ where $\leqslant$ is the specialization order over $X$;
    \item $\mathrm{CLOP}(X)$ is the collection of sets that are clopen in $X$;
     \item $\mathcal{CO}(X)=\mathcal{C}(X)\cap\mathcal{O}(X)$;
     \item $\mathcal{COR}(X)=\mathcal{CO}(X)\cap\mathcal{R}(X)$;
     \item $\mathcal{CO}\mathrm{RO}(X)=\mathcal{CO}(X)\cap\mathrm{RO}(X)$;
     \item $\mathrm{CLOP}\mathcal{R}(X)=\mathrm{CLOP}(X)\cap\mathcal{R}(X)$.
\end{enumerate}
\end{notation}

We will demonstrate that every ortholattice $L$ can be represented as $\mathcal{COR}(X)$ for some spectral space $X$.

Recall that a space $X$ is a $T_0$ space if $X$ satisfies the weakest separation axiom for topological spaces; namely, for points $x,y\in X$, if $x\not= y$, then there exists an open set $U\in\mathcal{O}(X)$ such that $x\in U$ and $y\not\in U$. A space $X$ is a \emph{compact} space if every basic open cover of $X$ admit of a finite subcover. A space $X$ is \emph{coherent} if $\mathcal{CO}(X)$ is closed under intersection and forms a basis for the topology over $X$. Lastly, a space $X$ is \emph{sober} if every completely prime filter in the lattice $\mathcal{O}(X)$ of open sets of $X$ is of the form:
 \[\mathcal{O}_X(x)=\{U\in\mathcal{O}(X)\mid x\in U\}\] for some point $x\in X$. We now recall the definition of a spectral space and then a classical instance of how spectral spaces arise.   

\begin{definition}\label{spectral space}
A topological space $X$ is a \textit{spectral space} if:
\begin{enumerate}
    \item $X$ is a $T_0$ space;\label{spectral space To}
    \item $X$ is a compact space;\label{spectral space compact} 
    \item $X$ is a coherent space; and\label{spectral space coherent}
    \item $X$ is a sober space.\label{spectral space sober} 
\end{enumerate}
\end{definition}
Recall that the spectrum of a commutative ring $R$ is given by $\text{spec}(R)=\{x\subseteq R\mid \text{$x$ is a prime ideal}\}$ endowed with the Zarski topology of closed sets of the form $\{x\in\text{spec}(R)\mid y\subseteq x\}$ for some ideal $y\subseteq R$.  

 \begin{theorem}[Hochster \cite{hochster}] A topological space $X$ is a spectral space if and only if $X$ is homeomorphic to $\text{spec}(R)$ for some commutative ring $R$.   
 \end{theorem}

The following results highlight the importance of spectral spaces for the purposes of the present article.

\begin{theorem}[Stone \cite{stone2}] Every distributive lattice can be represented (up to isomorphism) as $\mathcal{CO}(X)$ for some spectral space $X$.  
\end{theorem}

\begin{theorem}[Bezhanishvili and Holliday \cite{bezhanishvili}] Every Boolean algebra can be represented (up to isomorphism) as $\mathcal{CO}\mathrm{RO}(X)$ for some spectral space $X$. 

\end{theorem}

 Our representation theorem for ortholattices (in Theorem \ref{representation theorem}) is very much an analogue of the above representation theorem for Boolean algebras. 
\begin{definition}\label{spectral and patch topology}
 Let $L$ be an ortholattice, let $\mathfrak{F}(L)$ be the collection of all proper lattice filters of $L$, and define $\widehat{a}=\{x\in\mathfrak{F}(L)\mid a\in x\}$. Moreover, let $\perp_L\subseteq\mathfrak{F}(L)\times\mathfrak{F}(L)$ be the orthogonality relation defined by: \[x\perp_Ly\Longleftrightarrow\exists a\in L:a^{\perp}\in x\hspace{.1cm}\&\hspace{.1cm}a\in y.\] Then, we define the following topological spaces: 
 \begin{enumerate}
     \item $X^+_L=(X^+_L,\perp_L)$ is the space of proper lattice filters of $L$ whose topology is generated by $\{\widehat{a}\mid a\in L\}$, known as the \emph{spectral topology} over $X_L^+$.
 \item  $X_L^{\pm}=(X^{\pm}_L,\perp_L)$ is the space of proper lattice filters of $L$ whose topology is generated by $\{\widehat{a}\mid a\in L\}\cup\{\complement\widehat{a}\mid a\in L\}$ (where $\complement$ is the set-theoretic complement operator) known as the \emph{patch topology} over $X_L^{\pm}$.
 \end{enumerate}
 \end{definition}
 Note that $\widehat{a}\cap\widehat{b}=\widehat{a\wedge b}$ and so the subbasis $\{\widehat{a}\mid a\in L\}$ of the spectral topology for the space $X_L^+$ is closed under binary intersections. Moreover, note that since $\perp_L$ is an orthogonality relation over $\mathfrak{F}(L)$, $\perp_L$ is symmetric so for $x,y\in\mathfrak{F}(L)$, we can alternatively define $x\perp_Ly$ if and only if there exists some $a\in L$ such that $a\in x$ and $a^{\perp}\in y$. 

\subsection{The Stone space of an ortholattice}
Assuming Alexander's Subbase Theorem, it was shown in \cite{goldblatt1} that the space $X_L^{\pm}$ with its associated patch topology is a Stone space.
As demonstrated in the following proposition,
the use of some choice principle in this claim is essential.
\begin{proposition}\label{choice}
  The following are equivalent:
  \begin{enumerate}
  \item\label{choice1} PIT, the Prime Ideal Theorem for Boolean algebras.
  \item\label{choice2} The space $\xpm_L$ is compact for all Boolean algebras $L$.
  \end{enumerate}
\end{proposition}
\begin{proof}
  \newcommand{\subseteqfin}{\subseteq_{\mathrm{fin}}} To see that Condition \ref{choice}.\ref{choice1} implies Condition \ref{choice}.\ref{choice2}, note that the
  PIT proves the compactness of $\xpm_L$ for any Boolean algebra $L$
  as the only choice principle used in Goldblatt~\cite{goldblatt1} Alexander's Subbase Theorem,
  which is equivalent to PIT.

  To see that Condition~\ref{choice}.\ref{choice2} implies Condition~\ref{choice}.\ref{choice1}, assume that $\xpm_L$ is compact for all Boolean algebras $L$.
  To show PIT,
  it suffices \cite[Theorem 1]{howard}
  to prove: (1) the existence of a choice function for
  an arbitrary family of nonempty \emph{finite} sets, and (2) the Weak Rado Selection Lemma (whose statement can be found below).

  For the proof of the first statement, let $\mathcal S := (S_i)_{i \in I}$ be a family of nonempty finite sets.
  Let $L$ be the Boolean algebra presented by
  \[\langle \bigsqcup_{i \in I} S_i \mid \{a \wedge b = 0 \mid a \neq b \in S_i, i \in I\} \rangle.\]
  Consider $\xpm_L$.
  For $I' \subseteqfin I$,
  let $F_{I'} = \{ u \in \xpm_L \mid \zenbu{i \in I'} \aru{a \in S_i} a \in u\}$.
  It can be shown that $\mathcal F := (F_{I'})_{I' \in \powerfin(I)}$ is a filter basis
  of $\xpm_L$.
  Since $\xpm_L$ is compact,
  $\mathcal F$ has a cluster point $u^+$.
  We show that $f := \{(i, a) \mid i \in I, a \in S_i, a \in u^+\}$
  is a choice function for $\mathcal S$.
  Since $u^+$ is a proper filter of $L$,
  at most one $a \in S_i$ can belong to $u^+$ by the construction of $L$.
  This shows that $f$ is a function.
  We now show that $\dom f = I$.
  Let $i \in I$ be arbitrary.
  Suppose by way of contradiction that $S_i \cap u^+ = \0$.
  Then $\complement \widehat a$ is a neighborhood of $u^+$ for $a \in S_i$,
  and so is $U := \bigcap_{a \in S_i} \complement \widehat a$,
  which is open as $S_i$ is finite.
  Since $u^+$ is a cluster point, $U \cap F_{\{i\}}$ is nonempty,
  i.e., $\zenbu{a \in S_i} \aru{ u \in F_{\{i\}}} a \not \in u$,
  contradicting the definition of $F_{\{i\}}$.

  For the proof of the second statement, we will prove the Weak Rado Selection Lemma by showing the following: Suppose that for a set $\Lambda$ there is
    a family of functions $(\gamma_S)_{S \in \powerfin (\Lambda)}$
    such that $\gamma_S : S \to \{\pm 1\}$.
    Then there is $f : \Lambda \to \{\pm 1\}$ such that
    for all $S \subseteqfin \lambda$ there exists $T \subseteq \Lambda$
    with $S \subseteq T$ and $f \upharpoonright S = \gamma_T \upharpoonright S$.

  To that end, let $(\gamma_S)_S$ be given.
  Let a Boolean algebra $L$ be defined by using presentation:
  $L = \langle \lambda^+, \lambda^- \mid \lambda^+ = \neg \lambda^-\rangle_{\lambda \in \Lambda}$,
  where we have generators $\lambda^+, \lambda^-$
  corresponding to each $\lambda \in \Lambda$.
  For $S \subseteqfin \lambda$, let $u_S$ be the filter of $L$ generated by
  $\{\lambda^\pm \mid \lambda \in \Lambda, \gamma_S(\lambda) = \pm 1\}$.
  It can be shown that $u_S$ is proper so $u_S \in \xpm_L$.
  Consider the net $(u_S)_{S \in \powerfin(\Lambda)}$,
  where the indices are ordered by inclusion.
  Since $\xpm_L$ is compact,
  the net has a cluster point $u_\infty$.
  Now we have
  \newcommand{\thequantifiers}{\zenbu{\lambda \in \Lambda} \zenbu{S \subseteqfin \Lambda}\aru{T \supseteq S}}
  \[
    \thequantifiers
    [
      u_\infty \in \widehat{\lambda^\pm} \Rightarrow u_T \in \widehat{\lambda^\pm}
      \text{ and }
      u_\infty \in \complement\widehat{\lambda^\pm} \Rightarrow u_T \in \complement\widehat{\lambda^\pm}
      ],
    \]
    i.e.,
    \begin{equation}\label{eqchoice}
      \thequantifiers [\lambda^\pm \in u_\infty \iff \lambda^\pm \in u_T].
    \end{equation}
    Let $f = \{(\lambda, \pm 1) \mid \lambda^\pm \in u_\infty\}$.
    By a similar argument as before, $f$ is a function $\Lambda \to \{\pm 1\}$.
    Also, by \ref{eqchoice}, $\zenbu{S \subseteqfin \Lambda}\aru{T \supseteq S}
    f \upharpoonright T = \gamma_T$
    (\emph{a fortiori}, $f \upharpoonright S = \gamma_T \upharpoonright S$).
\end{proof}
\subsection{The representation theorem}
In contrast, we will demonstrate independently of Alexander's Subbase Theorem (along with its associated nonconstructive choice-principles) that the space $X_L^+$ with its associated spectral topology is a spectral space that represents (up to isomorphism) the original ortholattice $L$. We first verify that for every ortholattice $L$, the space $X_L^+$ gives rise to a spectral space. 

\begin{proposition}\label{lattice to space}
For every ortholattice $L$, the space $X_L^+$ is a spectral space whose specialization order $\leqslant$ is given by set-theoretic inclusion. 
 
\end{proposition}
\begin{proof}
To see that $X_L^+$ is a $T_0$ space, assume that $x,y\in X_L^+$ are such that $x\not=y$. If we then suppose without loss of generality that $a\in x\setminus y$, then $a\in x$ and $a\not\in y$ which implies that $x\in\widehat{a}$ and $y\not\in\widehat{a}$ where $\widehat{a}\in\mathcal{O}(X_L^+)$.

We now show that in fact $\widehat a$ is compact for each $a \in L$,
whence $\widehat 1 = X_L^+$ is also compact.
Since by definition of the space $X_L^+$, sets of the form $\widehat{a}$ are a basis for $X_L^+$,  it suffices to show that if $\widehat{a}\subseteq\bigcup_{i\in I}\widehat{b}_i$, then there exists a finite subcover. With that in mind, assume that $\widehat{a}\subseteq\bigcup_{i\in I}\widehat{b}_i$, then the principal filter $\uparrow \hspace{-.1cm}a=\{b\in L\mid a\leq b\}$ contains one of the $b_is$, which by the definition of principal filters implies that $a\leq b_i$ which means $\widehat{a}\subseteq \widehat{b}_i$, so $b_i$ is itself a finite subcover. 

To see that $X_L^+$ is a coherent space, first observe that by definition of $X_L^+$, it immediately follows that $\mathcal{CO}(X_L^+)$ forms a basis. To show that $\mathcal{CO}(X_L^+)$ is closed under binary intersections, let $U,V\in\mathcal{CO}(X_L^+)$. Then, observe that for finite index sets $I$ and $K$, we have $U=\bigcup_{i\in I}\widehat{a_i}$ and $V=\bigcup_{k\in K}\widehat{b_k}$ so \[U\cap V=\bigcup_{i\in I,k\in K}(\widehat{a_i}\cap\widehat{b_k})=\bigcup_{i\in I,k\in K}\widehat{a_i\wedge b_k}\] and thus $U\cap V$ is a finite union of compact open sets and therefore we have $U\cap V\in\mathcal{CO}(X^+_L)$. 

To show that $X_L^+$ is a sober space, it will be sufficient to show that every completely prime filter $x_p\subseteq\mathcal{O}(X_L^+)$ is of the form \[\mathcal{O}_{X_L^+}(x)=\{U\in\mathcal{O}(X_L^+)\mid x\in U\}\]
for some $x \in X_L^+ $. 
Hence, let $x$ be the filter in $L$ generated by the set $\{a\in L\mid \widehat{a}\in x_p\}$. 
The equality $x_p=\mathcal{O}_{X_L^+}(x)$ is can be seen by first observing that the inclusion $\mathcal{O}_{X_L^+}(x)\subseteq x_p$ is immediate by the definition of $x$ and $x_p$ being a filter. For the converse inclusion $x_p\subseteq\mathcal{O}_{X_L^+}(x)$, assume that $\bigcup_{i\in I}\widehat{a}_i\in x_p$. Since by hypothesis, $x_p$ is a completely prime filter, there exists some $a_i$ such that $\widehat{a}_i\in x_p$, which means that $a_i\in x$, hence $x\in\widehat{a}_i$. Therefore, we have that $\widehat{a}_i\in\mathcal{O}_{X_L^+}(x)$ so in particular, we have $\bigcup_{i\in I}\widehat{a}_i\in\mathcal{O}_{X_L^+}(x)$. Therefore, $X_L^+$ is a spectral space. 

Finally, note that since $X_L^+$ is a $T_0$ space, we have that for $x,y\in X_L^+$, $x\not\subseteq y$ implies that $x\not\leqslant y$. For the converse direction, suppose $x\subseteq y$. Then for each basic open $\widehat{a}$, if $x\in\widehat{a}$ i.e., $a\in x$, then $a\in y$ i.e., $y\in\widehat{a}$, which implies that $x\leqslant y$.   
\end{proof}

 Now that we have seen that given an ortholattice $L$, spaces of the form $X_L^+$ form a subclass of spectral spaces, we proceed to the promised choice-free representation theorem for ortholattices. 
\begin{theorem}\label{representation theorem}
Given an ortholattice $L$, the map $\widehat{\bullet}\colon L\to\mathcal{COR}(X_L^+)$ is an isomorphism where $\mathcal{COR}(X_L^+)$ is an ortholattice ordered by set-theoretic inclusion, whose operation for meet is $\cap$, whose operation for orthocomplement is $^*$, and whose bottom universal bound is $\emptyset$. 
\end{theorem}
 \begin{proof}
 We first show that the mapping $\widehat{\bullet}$ is an ortholattice homomorphism $L \to \mathcal R(X_L^+)$,
 where the codomain is known to be an ortholattice under the operations $\cap$, ${}^*$, and $\emptyset$
 \cite[Proposition 1]{goldblatt1}. We first check that $\widehat{\bullet}$ preserves meets by demonstrating $\widehat{a\wedge b}=\widehat{a}\cap\widehat{b}$. For the $\widehat{a\wedge b}\subseteq\widehat{a}\cap\widehat{b}$ inclusion, assume that $x\in\widehat{a\wedge b}$ so that $a\wedge b\in x$. Then, since $a\wedge b\leq a$ and $a\wedge b\leq b$, we have that $a\in x$ and $b\in x$ as $x$ is a filter. Hence, we find $x\in\widehat{a}$ and $x\in\widehat{b}$, so $x\in\widehat{a}\cap\widehat{b}$. For the $\widehat{a}\cap\widehat{b}\subseteq\widehat{a\wedge b}$ inclusion, assume that $x\in\widehat{a}\cap\widehat{b}$. Then, $x\in\widehat{a}$ and $x\in\widehat{b}$ so $a\in x$ and $b\in x$. Since $x$ is a filter, we find that $a\wedge b\in x$ and so $x\in\widehat{a\wedge b}$, as required. Hence, the function $\widehat{\bullet}$ is a homomorphism for $\wedge$.

We now verify that $\widehat{\bullet}$ preserves orthocomplements by demonstrating $\widehat{a^{\perp}}=(\widehat{a})^*$. For the $\widehat{a^{\perp}}\subseteq(\widehat{a})^*$ inclusion, suppose $x\in\widehat{a^{\perp}}$. Then $a^{\perp}\in x$ which implies that $x\perp_Ly$ for every $y\in\widehat{a}$ so $x\in (\widehat{a})^*$. 
For the inclusion $(\widehat{a})^*\subseteq\widehat{a^{\perp}}$, suppose that $x\in(\widehat{a})^*$ so $x\perp_Ly$ for every $y\in\widehat{a}$.  Let $y= \mathop\uparrow a=\{b\in L\mid b\geq a\}$ be the principal filter generated by $a\in L$. Then, we have $y=\uparrow\hspace{-.1cm}a\in\widehat{a}$ so $x\perp_Ly$. Hence, we have that there exists some $b\in L$ such that $b^{\perp}\in x$ and $b\in y$ i.e., $a\leq b$ which by Condition \ref{lattice2}.\ref{order reversing} implies that $b^{\perp}\leq a^{\perp}$. Therefore, we have that $a^{\perp}\in x$ i.e., $x\in \widehat{a^{\perp}}$, as required. Lastly, note that the equality $\widehat 0=\emptyset$ is obvious. 
 
 To show that $\widehat{\bullet}$ is an injection, let $a,b\in L$ such that $a\not=b$. If $a\not\leq b$, then $\uparrow\hspace{-.1cm}a\in \widehat{a}\setminus\widehat{b}$ which means $\widehat{a}\not=\widehat{b}$. 
 To see that the range of $\widehat\bullet$ is $\mathcal{COR}(X_L^+)$, suppose that $U\in\mathcal{COR}(X_L^+)$. Since $U$ is compact open, we have that
$U=\bigcup^n_{i=1}\widehat{a_i}$ for $a_1,\dots a_n\in L$, that is, $U$ is a finite union of basic opens. Since $U$ is also $\perp$-regular, we calculate  \[\widehat{\bigvee_{i=1}^na_i}=\Big(\bigcup_{i=1}^n\widehat{a_i}\Big)^{**}=U^{**}=U\] so $U$ is in the image of $\widehat{\bullet}$.
Since $\mathcal{COR}$ is the image of an ortholattice homomorphism,
$(\mathcal{COR}(X^+_L),\cap,^*,\emptyset)$ is an ortholattice.
\end{proof}

\section{The dual space of an ortholattice}
In describing topological spaces throughout this work, we will denote a general topological space by $X=(X,\mathcal{T})$ where $X$ is a set and $\mathcal{T}\subseteq\mathcal{P}(X)$ is some topology over $X$. Just as in our discussion of lattices, we will often conflate a topological space with its underlying carrier set. We proceed by characterizing the class of spectral spaces which are homeomorphic to the space $X_{L}^+$ for some ortholattice $L$.

\subsection{UVO-spaces}

The following definition is an analogue of the construction given in \cite{bezhanishvili} of the class of spectral spaces which are homeomorphic to the space $X^+_B$ for some Boolean algebra $B$. 

\begin{definition}\label{uvo space}
Let $X=(X,\leqslant,\bot,\mathcal{T})$ be an ordered topological space endowed with an orthogonal binary relation $\bot\subseteq X^2$ and whose specialization order is $\leqslant$, then $X$ is an \emph{upper Vietoris orthospace} (henceforth, a \textit{UVO-space}) whenever the following conditions are satisfied: 
\begin{enumerate}
\item\label{uvoT0} $X$ is a $T_0$ space;
\item\label{uvo closure} $\mathcal{COR}(X)$ is closed under $\cap$ and $^*$;
\item\label{uvo basis} $\mathcal{COR}(X)$ is a basis for $X$;
    \item\label{uvo proper filter} Every proper filter in $\mathcal{COR}(X)$ is of the form: \[\mathcal{COR}_X(x)=\{U\in\mathcal{COR}(X)\mid x\in U\}\]
    for some $x\in X$; and
    \item\label{uvo perp} $x\perp y\Longrightarrow \exists U\in\mathcal{COR}(X): x\in U\hspace{.1cm}\&\hspace{.1cm}y\in U^*$
\end{enumerate}
\end{definition}
 Note that given a UVO-space $X$, the requirement that $\mathcal{COR}(X)$ form a basis for $X$ implies the following analogue of the Priestly separation axiom: \[x\not\leqslant y\Longrightarrow\exists U\in\mathcal{COR}(X):x\in U\hspace{.1cm}\&\hspace{.1cm} y\not\in U.\]  Notice that if we replace the compact open $\perp$-regular subsets of $X$ by the clopen upsets of $X$, then we arrive exactly at Priestley's seperation for the dual space of a bounded distributive lattice. Moreover, note that the fourth condition is an analogue of the sobriety condition of a spectral space. 
 
 The construction which associates to each UVO-space $X$ an ortholattice $L$ is provided to us by the following lemma. 

\begin{lemma}\label{space to lattice}
If $X$ is a UVO-space, then $L=(\mathcal{COR}(X),\cap,^*,\emptyset)$ is an ortholattice. 
\end{lemma}
 \begin{proof}Here, we define the joins of $L$ by De Morgan's distribution laws for complements over meets and set $1=\emptyset^*$. We first verify that $\mathcal{COR}(X)$ is closed under the relevant operations. Clearly $\emptyset\in\mathcal{CO}(X)$ and since $\emptyset=\emptyset^{**}$, we have that $\emptyset\in\mathcal{COR}(X)$. By Condition \ref{uvo space}.\ref{uvo closure}, if $U\in\mathcal{COR}(X)$ then $U^*\in\mathcal{COR}(X)$ and if $U,V\in\mathcal{COR}(X)$, then $U\cap V\in\mathcal{COR}(X)$. 
 
 To see that the algebra induced by $\mathcal{COR}(X)$ is an ortholattice, first observe that by the irreflexivity of $\perp$, we have that $U\cap U^*=\emptyset$ for every $U\in\mathcal{COR}(X)$. If on the other hand there was some $y\in U\cap U^*$, then by definition of $U^*$, we would have $y\in\{x\mid \forall y\in U:x\bot y\}$ which contradicts the fact that $\bot$ is irreflexive. Hence Condition \ref{lattice2}.\ref{complement} is satisfied. Given the definition of the $^*$ operator, the symmetry of $\bot$ guarantees that $^*$ is an order-reversing function, so Condition \ref{lattice2}.\ref{order reversing} is satisfied. By the $\perp$-regularity of $\mathcal{COR}(X)$, if $U\in\mathcal{COR}(X)$, then $U=U^{**}$ so Condition \ref{lattice2}.\ref{double complement} is satisfied.    
 \end{proof}
 We now must conversely verify that every ortholattice gives rise to a UVO-space.   
\begin{lemma}\label{lattice to uvo-space}
 If $L$ is an ortholattice, then $X_L^+=(X^+_L,\perp_L)$ is a UVO-space.  
\end{lemma}
\begin{proof}
 We first verify that $\perp_L\subseteq \mathfrak{F}(L)\times\mathfrak{F}(L)$ is indeed an orthogonality relation over the proper filters of $L$. For irreflexivity, assume by contradiction that there exists $x\in\mathfrak{F}(L)$ such that $x\perp_L x$. Then, there exists $a^{\perp}\in x$ such that $a\in x$. Since $x$ is a filter, we have that $a\wedge a^{\perp}\in x$ which by Condition \ref{lattice2}.\ref{complement} implies that $0\in x$ which contradicts the fact that $x$ is a proper lattice filter over $L$. Therefore, $\perp_L$ is irreflexive. For symmetry assume that $x,y\in\mathfrak{F}(L)$ are such that $x\perp_L y$. Then by definition, there exists $a^{\perp}\in x$ such that $a\in y$. By Condition \ref{lattice2}.\ref{double complement}, we have that $a^{\perp\perp}=a$ and so $a^{\perp\perp}\in y$ but since $a^{\perp}\in x$, we have that $y\perp_L x$ by the definition of $\perp_L$. Hence, we conclude that $\perp_L$ is symmetric. 

We already know that $X_L^+$ is a $T_0$ space from Proposition \ref{lattice to space}. Note that by Theorem \ref{representation theorem}, if $U,V\in\mathcal{COR}(X^+_L)$, then $U=\widehat{a}$ and $V=\widehat{b}$ for some $a,b\in L$. Moreover, we saw that $\widehat{a}\cap\widehat{b}=\widehat{a\wedge b}$ and $(\widehat{a})^*=\widehat{a^{\perp}}$ with $\widehat{a\wedge b}\in\mathcal{COR}(X^+_L)$ and $\widehat{a^{\perp}}\in\mathcal{COR}(X^+_L)$. Since by definition, sets of the form $\widehat{a}$ for some $a\in L$ form a basis for the space $X^+_L$, it follows that the second and third conditions are satisfied. For the fourth condition, let $x$ be a proper filter in $\mathcal{COR}(X_L^+)$. Then $y=\{a\in L\mid \widehat{a}\in x\}$ is a proper filter in $L$. It is easy to verify that $\mathcal{COR}_{X_L^+}(y)=x$.
Finally, for the fifth condition, let $x,y\in\mathfrak{F}(L)$ such that $x\perp_Ly$. Then there exists some $a\in L$ such that $a\in x$ and $a^{\perp}\in y$. By the definition of $\widehat{a}$, we have that $x\in\widehat{a}$ and that $y\in\widehat{a^{\perp}}$, but since $\widehat{\bullet}$ is a homomorphism for $^{\perp}$, we have that $y\in(\widehat{a})^*$. Again, by Theorem \ref{representation theorem}, for $U\in\mathcal{COR}(X^+_L)$, we have $U=\widehat{a}$ for some $a\in L$, which means that there exists some $U\in\mathcal{COR}(X^+_L)$ such that $x\in U$ and $y\in U^*$, as desired.
\end{proof}

 \subsection{The characterization theorem for $X^+_L$}
 
 We now proceed by demonstrating that the class of UVO-spaces provides us with the desired topological characterization of the class of spectral spaces used in our representation.     
  
\begin{theorem}\label{characterization theorem} For each UVO-space $X$, the map $X\to X_{\mathcal{COR}(X)}^+$ is a homeomorphism and an isomorphism  with respect to the orthospace reducts (X,$\perp$) and $(X^+_{\mathcal{COR}(X)},\perp)$.   
\end{theorem}
\begin{proof}
We will show that the map $g\colon x\mapsto\mathcal{COR}_X(x)$ gives the desired homeomorphism from $X$ to $X_{\mathcal{COR}(X)}^+$. To see that $g$ is an injective function, let $x,y\in X$ and assume that $x\not= y$. Since $X$ is a $T_0$ space, we have that either $x\not\leqslant y$ or $y\not\leqslant x$. If $x\not\leqslant y$, then from Condition \ref{uvo space}.\ref{uvo basis} (which, as already mentioned, implies our analogue of the Priestly separation axiom), we have that there exists some $U\in\mathcal{COR}(X)$ such that $x\in U$ and $y\not\in U$, which implies that $U\in\mathcal{COR}_X(x)$ and $U\not\in\mathcal{COR}_X(y)$ so we have the desired inequality $\mathcal{COR}_X(x)\not=\mathcal{COR}_X(y)$. If on the other hand, we have that $y\not\leqslant x$, then we similarly find that there exists some $U\in\mathcal{COR}(X)$ such that $y\in U$ but $x\not\in U$, which implies that $\mathcal{COR}_X(x)\not=\mathcal{COR}_X(y)$.  As the surjectivity of $g$ is immediate from Condition \ref{uvo space}.\ref{uvo proper filter}, we have established that $g$ is a bijective function.

To see that $g$ is continuous, it will suffice to demonstrate that the inverse image of each basic open set in $X^+_{\mathcal{COR}(X)}$ is an open set in $X$. Note that each basic open set in $X_{\mathcal{COR}(X)}^+$ is of the form $\widehat{U}$ for some $U\in X_{\mathcal{COR}(X)}^+$. The continuity of $g$ can then be proved by observing the following calculation:
\begin{align*}
    g^{-1}[\widehat{U}]&=\{x\in X\mid \mathcal{COR}_X(x)\in\widehat{U}\}
    \\&=\{x\in X\mid U\in\mathcal{COR}_X(x)\}
    \\&=\{x\in X\mid x\in U\}
    \\&=U.
\end{align*}
\noindent The continuity of $g^{-1}$ is established by calculating the image of $g$ as follows: 
\begin{align*}
    g[\widehat{U}]&=\{\mathcal{COR}_X(x)\mid x\in U\}\\&=\{\mathcal{COR}_X(x)\mid U\in\mathcal{COR}_X(x)\}\\&=\widehat{U}.
\end{align*}

 Now that we have established that $g$ is a homeomorphism of topological spaces, we proceed by verifying that $g$ is an isomorphism with respect to the orthospace reducts. Suppose for $x,y\in X$, we have $g(x)\perp g(y)$. Then by the definition of $g$, we have that $\mathcal{COR}_X(x)\perp\mathcal{COR}_X(y)$. By the definition of $\perp$, this implies that there exists some $U\in\mathcal{COR}(X)$ such that $U\in\mathcal{COR}_X(x)$ and $U^*\in\mathcal{COR}_X(y)$ which means that $x\in U$ and $y\in U^*$. By universal instantiation and the definition of the $^*$ operator, we have that $x\perp y$. Conversely, let $x,y\in X$ and suppose that $x\perp y$. By hypothesis, $X$ is a UVO-space and so by Condition \ref{uvo space}.\ref{uvo perp}, there exists some $U\in\mathcal{COR}(X)$ such that $x\in U$ and $y\in U^*$. By the definition of $g$, this means that $U\in\mathcal{COR}_X(x)$ and $U^*\in\mathcal{COR}_X(y)$. Hence, by the definition of $\perp$, we have that $\mathcal{COR}_X(x)\perp\mathcal{COR}_X(y)$ i.e., $g(x)\perp g(y)$.      
\end{proof}

\begin{corollary}\label{corollary to characterization theorem}
Let $X$ be a UVO-space. Then:
\begin{enumerate}
    \item $X$ is a spectral space.
    \item Every element in $\mathcal{CO}(X)$ is a finite union of elements in $\mathcal{COR}(X)$.
\end{enumerate}
\end{corollary}
\begin{proof}
 For part 1, note that by Theorem \ref{characterization theorem}, we have that every UVO-space $X$ is homeomorphic to the space $X_{\mathcal{COR}(X)}^+$, which is a spectral space by Proposition \ref{lattice to space}, since $\mathcal{COR}(X)$ is an ortholattice whenever $X$ is a UVO-space by Lemma \ref{space to lattice}. For part 2, let $X$ be a UVO-space and let $U\in\mathcal{CO}(X)$. Then by Condition \ref{uvo space}.\ref{uvo basis}, $U$ is a finite union of elements from $\mathcal{COR}(X)$.     
\end{proof}
\section{The category of UVO-spaces}
\noindent We now proceed by investigating the abstract category-theoretic structure underlying the constructions and results achieved in the previous two sections. For an in-depth exposition of pure category theory, refer to \cite{awodey}.  

\begin{definition}\label{category of ortholattices}
Let $\mathbf{OrthLatt}$ be the category whose collection of objects are given by the class of ortholattices and whose collection of morphisms are given by the class of ortholattice homomorphisms between them.  
\end{definition}
It is clear that isomorphisms in the category $\mathbf{OrthLatt}$ coincide with algebraic isomorphisms.
\subsection{UVO-mappings}
 Just as in the categorical dual equivalence result in \cite{bezhanishvili} between the category $\mathbf{BoolAlg}$ of Boolean algebras and Boolean homomorphisms and the category $\textbf{UV}$ of UV-spaces and UV-mappings, our conception of an appropriately defined continuous function between UVO-spaces depends upon the notions of a spectral mapping and a weak \textit{p}-morphism. 

\begin{definition}\label{spectral map}
Given spectral spaces $X$ and $X'$, a map $f\colon X\to X'$ is a \textit{spectral map} if $f^{-1}[U]\in\mathcal{CO}(X)$ for every $U\in\mathcal{CO}(X')$.
\end{definition}
\begin{example}
If $X$ is a spectral space and $Y$ is a Stone space, then the map $f\colon X\to Y$ is a spectral map if and only if $f$ is a continuous function.
\end{example}

Clearly, any spectral map is a continuous function but the converse is not in general true, as can be easily seen through the following example.  
\begin{example}
Let $X$ be an infinite Stone space and let $Y$ be the Sierpi\'nsky space; namely the topological space whose carrier set is the two-element set $\{0,1\}$ and whose topology is generated by $\{\{1\}\}$.
Note that $\{1\}$ is compact open.
Take a non-isolated point $x\in X$.
Then $\complement\{x\}$ is not compact.
The characteristic function $f\colon X\to Y$ of this open set is continuous but not spectral since $f^{-1}(\{1\}) = \complement\{x\}$ is not compact.     
\end{example}

\begin{definition}\label{weak p-morpohism}
  Let $(X, \bot)$ and $(X, \bot')$ be UVO-spaces and $f: X \to X'$ be a function.
  Such a function is \emph{weakly p-morphic}, or a \emph{weak p-morphism},
  if it is a homomorphism with respect to the relations $\not\perp$
  and $\not\perp'$,
  and
  for every $y \in X$ and $z \in X'$, if $z \not \perp f(y)$, then
  there exists $x \in X$ such that $x \not \perp y$ and that $z \leqslant f(x)$.
\end{definition}
Note that a function is weakly $p$-morphic
whenever it is $p$-morphic with respect to the complements of the orthogonality
relations.

With the notions of a spectral map and a weak $p$-morphism in mind, we arrive at the notion of a UVO-map. 
\begin{definition}\label{uvo-map}
If $X$ and $X'$ are UVO-spaces, then a map $f\colon X\to X'$ is a \textit{UVO-map} if $f$ is a spectral map and a weak $p$-morphism.
\end{definition}
The above construction of a UVO-map between UVO-spaces is highly reminscent to the construction of a continuous map between two Stone spaces of an ortholattice, as defined in \cite{bimbo}.

 \begin{definition}\label{homeomorphic uvo map}
 If $X$ and $Y$ are UVO-spaces, then a UVO-map $f\colon X\to Y$ is a \emph{homeomorphism} if $f$ is a homeomorphism as a continuous map and an isomorphism (or bijective embedding) with respect to the orthospace (orthoframe) reducts $(X,\perp)$ and $(Y,\perp)$ of $X$ and $Y$ respectively.    
 \end{definition}

\begin{definition}
Let $\textbf{UVO}$ be the category whose collection of objects are given by the class of UVO-spaces and whose collection of morphisms are given by the class UVO-mappings between them. 
\end{definition}
Note that the isomorphisms in $\mathbf{UVO}$ are given exactly by those UVO-maps which are homeomorphisms as described in Definition \ref{homeomorphic uvo map}.  
\subsection{Basic results about UVO-mappings}
 The following results will be useful in our proof of the categorical dual equivalence between $\textbf{OrthLatt}$ and $\textbf{UVO}$. 
 \begin{proposition}\label{prop1 spectral map}
 If $X$ and $X'$ are UVO-spaces and $f\colon X\to X'$ is a UVO-map, then $f^{-1}[U]\in\mathcal{COR}(X)$ for each $U\in\mathcal{COR}(X')$.  
 \end{proposition}
 \begin{proof}
   Since $f$ is spectral by hypothesis,
  the inverse image $f^{-1}[U]$ of such $U$ is compact open.  It was proved by Bimb\'o~\cite[Lemma 3.9]{bimbo} within ZF that $f^{-1}[U]$ for a regular $U$ is again regular.
\end{proof}
\begin{proposition}\label{prop2 spectral map}
If $X$ and $X'$ are UVO-spaces and $f\colon X\to X'$ is a map such that $f^{-1}[U]\in\mathcal{CO}(X)$ for every $U\in\mathcal{COR}(X')$, then $f$ is a spectral map. 
\end{proposition}
\begin{proof}
Suppose that $X$ and $X'$ are UVO-spaces and that $f\colon X\to X'$ is a UVO-map. Then by Corollary \ref{corollary to characterization theorem}, we find that $U=\bigcup_{i=1}^nU_i$ for $U_i\in\mathcal{COR}(X')$, which yields the following equalities: \[f^{-1}[U]=f^{-1}\Bigg[\bigcup_{i=1}^nU_i\Bigg]=\bigcup_{i=1}^nf^{-1}[U_i].\] By hypothesis, we have that $f^{-1}[U_i]\in\mathcal{CO}(X)$ which implies that $f^{-1}[U]$ is a finite union of compact opens and thus $f$ is a spectral map. 
\end{proof}
\begin{lemma}\label{lemma1 spectral map}
Let $X$ and $X'$ be spectral spaces and let $f\colon X\to X'$ be a map. If for each set $U$ in some subbasis of $X'$, we have $f^{-1}[U]\in\mathcal{CO}(X)$, then $f$ is a spectral map. 
\end{lemma}
\begin{proof}
   By definition, every open set $U\in\mathcal{O}(X)$ is a union of finite intersections of subbasic open sets so every compact open set $U\in\mathcal{CO}(X)$ is a finite union $\bigcup^n_{i=1}U_i$ of finite intersections of subbasic sets. Then, since \[f^{-1}[U]=f^{-1}\Bigg[\bigcup_{i=1}^nU_i\Bigg]=\bigcup_{i=1}^nf^{-1}[U_i]\] it follows that $f^{-1}[U]\in\mathcal{CO}(X)$ if every $f^{-1}[U_i]\in\mathcal{CO}(X)$. Moreover, given that  $U_i=\bigcap_{j=1}^mV_j$ where each $V_j$ is a subbasic set and given that \[f^{-1}[U_i]=f^{-1}\Bigg[\bigcap_{j=1}^mV_j\Bigg]=\bigcap_{j=1}^mf^{-1}[V_j]\] it similarly follows that $f^{-1}[U_i]\in\mathcal{CO}(X)$ if every $f^{-1}[V_j]\in\mathcal{CO}(X)$. Finally, since by hypothesis, the inverse image of each $V_j$ is compact open, we have that $f$ is a spectral map, as desired.     
\end{proof}

\subsection{The main result}

 We now proceed with the promised choice-free categorical dual equivalence result between the categories $\mathbf{OrthLatt}$ and $\mathbf{UVO}$.     

\begin{theorem}\label{main theorem}
The category $\textbf{OrthLatt}$ of ortholattices and ortholattice homomorphisms and the category $\textbf{UVO}$ of UVO-spaces and UVO-mappings constitute a dual equivalence of categories. 
\end{theorem}
\begin{proof}
 Let $L$ and $L'$ be ortholattices and let $h\colon L\to L'$ be an ortholattice homomorphism. Given $x\in X_{L}^+$, define $h_+(x)=h^{-1}[x]$. Since $h$ is an ortholattice homomorphism, $h_+(x)$ is a proper lattice filter in $L$. Hence, we have an induced map $h_+\colon X_{L'}^+\to X_{L}^+$. We want to show that $h_+$ is a UVO-map. We first verify that $h_+$ is a spectral map. By Lemma \ref{lemma1 spectral map}, it will suffice to show that for each basic open $\widehat{a}$ in the space $X_{L}^+$, we have that $h_+^{-1}[\widehat{a}]\in\mathcal{CO}(X_{L}^+)$. This is achieved by observing the following calculation:
\begin{align*}
    h_+^{-1}[\widehat{a}]&=\{x\in X_{L'}^+\mid h_+(x)\in\widehat{a}\}
    \\&=\{x\in X_{L'}^+\mid h^{-1}[x]\in\widehat{a}\}
    \\&=\{x\in X_{L'}^+\mid a\in h^{-1}[x]\}
    \\&=\{x\in X_{L'}^+\mid h(a)\in x\}
    \\&=\widehat{h(a)}.
\end{align*}
This is compact open.

It can be proved that $h_+$ is weakly $p$-morphic in the same way as Bimb\'o~\cite[Lemma 3.9]{bimbo}.

For the other direction, suppose that $X$ and $X'$ are UVO-spaces and that $f\colon X\to X'$ is a UVO-map. Given any $U\in\mathcal{COR}(X')$, define $f^+(U')=f^{-1}[U]$. Note that by Proposition \ref{prop1 spectral map}, we have that $f^+(U)=f^{-1}[U]\in\mathcal{COR}(X)$ since $f$ is by hypothesis a UVO-map.
It can be proved that $h_+$ is an ortholattice homomorphism in the same way as Bimbo~\cite[Lemma 3.10]{bimbo}.

Clearly $(\bullet)^+$ preserves identity maps and the composition structure. Hence $(\bullet)^+$, $\mathcal{COR}(\bullet)$, along with Lemmas \ref{lattice to uvo-space} and \ref{space to lattice}, give rise to contravariant functors  $(\bullet)^+\colon\textbf{OrthLatt}\to\textbf{UVO}$ and $\mathcal{COR}(\bullet)\colon\textbf{UVO}\to\textbf{OrthLatt}$
where $(\bullet)^+$ is defined on objects and morphisms by \[L\mapsto X^+_L,\hspace{.2cm}h\colon L\to L'\mapsto h_+\colon X^+_{L'}\to X^+_L\] and $\mathcal{COR}(\bullet)$ is defined on objects and morphisms by \[X\mapsto\mathcal{COR}(X),\hspace{.2cm}f\colon X\to X'\mapsto f^+:\mathcal{COR}(X')\to\mathcal{COR}(X).\]

\noindent In light of Theorem \ref{representation theorem} which established that every ortholattice $L$ is isomorphic to $\mathcal{COR}(X^+_L)$, it is not difficult to verify that every ortholattice homomorphism $h\colon L\to L'$ makes the following diagram commute: 
\[
\begin{tikzcd}
L \arrow{r}{h} \arrow[swap]{d} & L' \arrow{d} \\
\mathcal{COR}(X_L^+)  \arrow[swap]{r}{(h_+)^+} & \mathcal{COR}(X_{L'}^+)
\end{tikzcd}
\]
which implies that each component
$\eta_L\colon 1_L(L)\to \mathcal{COR}(\bullet)\circ (\bullet)^+(L)$ of the natural transformation $\eta\colon 1_{\mathbf{OrthLatt}}\to\mathcal{COR}(\bullet)\circ(\bullet)^+$ is an isomorphism.

Similarly, in light of Theorem \ref{characterization theorem}, which established that every UVO-space $X$ is homeomorphic to $X^+_{\mathcal{COR}(X)}$ and order isomorphic with respect to the complements of the orthogonality relations, it is not difficult to verify that every UVO-map $f\colon X\to X'$ makes the below diagram commute:
\[
\begin{tikzcd}
X \arrow{r}{f} \arrow[swap]{d} & X' \arrow{d} \\
X_{\mathcal{COR}(X)}^+  \arrow[swap]{r}{(f^+)_+} & X_{\mathcal{COR}(X')}^+
\end{tikzcd}
\]
which implies that each component
$\theta_X\colon 1_X(X)\to (\bullet)^+\circ\mathcal{COR}(\bullet)(X)$ of the natural transformation $\theta:1_{\mathbf{UVO}}\to(\bullet)^+\circ\mathcal{COR}(\bullet)$
is an isomorphism, which completes our proof that the contravariant functors $\mathcal{COR}(\bullet)$ and $(\bullet)^+$ constitute a dual equivalence of categories.
\end{proof}

\section{Duality dictionary}

\noindent In light of Theorem \ref{main theorem}, we proceed by developing a \say{duality dictionary} (as depicted in the Figure \ref{duality dictionary}) for the purposes of explicitly establishing how one can translate between various lattice-theoretic concepts (as applied to the category $\mathbf{OrthLatt}$) and their corresponding dual topological concepts in the category $\mathbf{UVO}$. For an analogous duality dictionary relating the category of Boolean algebras $\mathbf{BoolAlg}$, the category of UV-spaces $\mathbf{UV}$, and the category of Stone spaces $\mathbf{Stone}$, refer to \cite{bezhanishvili}.     

\begin{figure}[htbp]
\begin{tabularx}{1\textwidth} { 
  | >{\centering\arraybackslash}X 
  | >{\centering\arraybackslash}X 
  | }
 \hline
 \textbf{OrthLatt} & \textbf{UVO}  \\
 \hline
 ortholattice  & UVO-space    \\
\hline homomorphism & UVO-map \\
\hline complete lattice & complete UVO-space \\
\hline atom & isolated point  \\
\hline atomless lattice & $X_{\text{iso}}=\emptyset$      \\
\hline atomic lattice & $\text{Cl}(X_{\text{iso}})=X$  \\
\hline injective homomorphism & surjective UVO-map  \\
\hline surjective homomorphism & UVO-embedding  \\
\hline subalgebra & image under UVO-map \\
\hline direct product & UVO-sum\\

\hline MacNeille completion & $\mathcal{R}(\mathfrak{P}(X))$ \\
\hline canonical extension & $\mathcal{R}(X)$   \\
\hline 
  \end{tabularx}
\caption{Duality dictionary for $\mathbf{Orthlatt}$ and $\mathbf{UVO}$}\label{duality dictionary}
\end{figure}

\subsection{Complete lattices}

\begin{definition}
Let $X$ be a UVO-space, then $X$ is \textit{complete} if for every open set $U\in\mathcal{O}(X)$, we have that $U^{*\circ*}\in\mathcal{COR}(X)$.%
\footnote{We sometimes write $(\bullet)^\circ$ for the interior of a set where appropriate in terms of typography.}
\end{definition}
We now verify that complete UVO-spaces and the duals of complete ortholattices coincide.    
\begin{proposition}\label{complete lattice}
Let $L$ be an ortholattice and let $X$ be its dual UVO-space. Then:
\begin{enumerate}
    \item\label{complete 1} A family $\{U_i\}_{i\in I}\subseteq\mathcal{COR}(X)$ has a greatest lower bound in $\mathcal{COR}(X)$ iff $\text{Int}\big(\bigcap_{i\in I}U_i \big)\in\mathcal{COR}(X)$, in which case \[\bigwedge_{i\in I}U_i=\text{Int}\Big(\bigcap_{i\in I}U_i\Big).\]
    \item\label{complete 2} A family $\{U_i\}_{i\in I}\subseteq\mathcal{COR}(X)$ has a least upper bound in $\mathcal{COR}(X)$ iff $(\bigcup_{i\in I}U_i)^{*\circ *}\in\mathcal{COR}(X)$, in which case \[\bigvee_{i\in I}U_i=\big(\bigcup_{i\in I}U_i\big)^{*\circ *}.\]
     \item\label{complete 3} $L$ is a complete ortholattice iff $X$ is a complete UVO-space.
\end{enumerate}
\end{proposition}
\begin{proof}
For part 1, observe that $\text{Int}\big(\bigcap_{i\in I}U_i\big)=\inf(\{U_i\}_{i\in I})$ for $\{U_i\}_{i\in I}\subseteq\mathcal{COR}(X)$ immediately follows from our hypothesis that $\mathrm{Int}(\bigcap_{i\in I}U_i)\in\mathcal{COR}(X)$. The for left to right implication of part 1,  assume that $\bigwedge_{i\in I}U_i$ is defined in $\mathcal{COR}(X)$. Note that by Theorem \ref{representation theorem}, for every $i\in I$, there exists some $\widehat{a_i}\in L$ such that $U_i=\widehat{a_i}$, and since the map $\widehat{\bullet}\colon L\to\mathcal{COR}(X^+_L)$ defined by $a\mapsto\widehat{a}$ is an ortholattice isomorphism, we have the following equalities:
\[\bigwedge_{i\in I}U_i=\bigwedge_{i\in I}\widehat{a_i}=\widehat{\bigwedge_{i\in I}a_i}.\] Hence, it suffices to show that \begin{equation}
    \widehat{\bigwedge_{i\in I}a_i}=\text{Int}\Big(\bigcap_{i\in I}\widehat{a_i}\Big).
\end{equation}  

We have $\widehat{\bigwedge_{i\in I}a_i}\subseteq\text{Int}\Big(\bigcap_{i\in I}\widehat{a_i}\Big)$ as clearly $\widehat{\bigwedge_{i\in I}a_i}\subseteq\bigcap_{i\in I}\widehat{a_i}$ such that $\widehat{\bigwedge_{i\in I}a_i}$ is an open set.
To see that $\text{Int}\Big(\bigcap_{i\in I}\widehat{a_i}\Big)\subseteq\widehat{\bigwedge_{i\in I}a_i}$, suppose that $x\in\text{Int}\Big(\bigcap_{i\in I}\widehat{a_i}\Big)$. Then there exists some $U\in\mathcal{COR}(X)$ such that $x\in U\subseteq\bigcap_{i\in I}\widehat{a_i}$. Hence, by Theorem \ref{representation theorem}, we have that $U=\widehat{b}$ for some $b\in L$. Moreover, since $\widehat{b}\subseteq\bigcap_{i\in I}\widehat{a_i}$, it follows that $b\leq\bigwedge_{i\in I}a_i$. Then, since $x\in\widehat{b}$, we have $b\in x$ so $\bigwedge_{i\in I}a_i\in x$, hence $x\in\widehat{\bigwedge_{i\in I}a_i}$.

For part 2,
we first assume that $\bigvee_i U_i$ exists.
Let $a_i \in L$ be such that $\widehat{a_i} = U_i$.
One can show that $u \in (\bigcup_i \widehat{a_i})^{*\circ*}$
if and only if
\[
\zenbu{v\in X}[\underbrace{\aru{b\in v}\zenbu{i \in I} a_i \le b^\bot}_{(\dagger)} 
\implies u \perp v].
\]
With this in mind, let $u \in \widehat{\bigvee_i U_i}$ be arbitrary.
We show $u \in (\bigcup_i \widehat{a_i})^{*\circ*}$.
Take an arbitrary $v \in X$ with (\textdagger).
Since $a_i \le b^\bot$ for all $i \in I$, we have $\bigvee_i a_i \le b^\bot$.
The left-hand side of this inequality is in $u$ by assumption, 
as is the right-hand side.
Recall $b \in v$ to conclude $u \perp v$.
We have established $u \in (\bigcup_i \widehat{a_i})^{*\circ*}$.
To show the other inclusion,
we prove $u \not\in (\bigcap_i \widehat{a_i})^{*\circ*}$
for $u \not \in \widehat{\bigvee_i a_i}$.
It suffices to exhibit $v \in X$ with the properties (\textdagger)
and $u\not\perp v$.
Let $b = (\bigvee_i a_i)^\bot$ and $v = \mathop\uparrow b$.
Since $b^\bot = \bigvee_i a_i$, the property (\textdagger) is satisfied.
Now, since $v$ is a principal filter, 
$u \perp v$ if and only if $b^\bot \in u$,
and we assumed otherwise.
Now we show that if $(\bigcup_i \widehat{a_i})^{*\circ*} \in \mathcal{COR}(X)$,
then $(\bigcup_i \widehat{a_i})^{*\circ*}$ 
is the least upper bound of $\{U_i\}_i$.
Clearly, it is an upper bound of the family,
so it suffices to show that if $\widehat c$ is an upper bound of the family,
then $(\bigcup_i \widehat{a_i})^{*\circ*} \subseteq \widehat c$.
Take an arbitrary $u$ in the left-hand side of the inequality.
Let $b = c^\bot$ and $v = \mathop\uparrow b$.
Since $c = b^\bot$, and $a_i \le c$ for all $i \in I$,
the property (\textdagger) is satisfied.
We conclude $u \perp v$, i.e., $c \in u$ as desired.

For part 3, we start by proving the left-to-right implication. Assume $L$ is a complete ortholattice so that for each $A\subseteq L$, we have that $\bigwedge A$ and $\bigvee A$ are defined. If $U\in \mathcal{O}(X)$, then by Definition \ref{spectral space}, we have that \[U=\bigcup\{V\in\mathcal{COR}(X)\mid V\subseteq U\}.\]  Since by hypothesis, $L$ is a complete ortholattice, by Theorem \ref{representation theorem}, so is the corresponding unique (up to isomorphism) ortholattice induced by $\mathcal{COR}(X)$ and hence $\bigvee\{V\subseteq\mathcal{COR}(X)\mid V\subseteq U\}$ exists. By our proof of part 2, we have 
\[
\bigvee\{V\subseteq\mathcal{COR}(X)\mid V\subseteq U\}=
\big(\bigcup\{V\subseteq\mathcal{COR}(X)\mid V\subseteq U\}\big)^{*\circ*},
\] 
which implies that $U^{*\circ*}\in\mathcal{COR}(X)$ as desired,
making $X$ a complete UVO-space by definition.
Conversely, suppose that $X$ is a complete UVO-space. Then for every family of subsets $\{U_i\}_{i\in I}\subseteq L$, we have $\bigcup_{i\in I}\widehat{a_i}\in\mathcal{O}(X)$. Since by hypothesis, $X$ is a complete UVO-space and so we have that $\big(\bigcup_{i\in I}\widehat{a_i}\big)^{*\circ*}\in\mathcal{COR}(X)$. 
By part 2, it follows that $\bigvee_{i\in I}a_i$ exists.
Finally, recall that since the orthocomplement operation of $L$ is an
isomorphism between $L$ and the order dual of $L$,
the ortholattice $L$ is complete if and only if arbitrary joins exist in $L$.
We conclude that $L$ is complete if $X$ is complete. 
\end{proof}

\subsection{Atoms}
\begin{notation}
Let $L$ be an ortholattice and let $X$ be a UVO-space. We write $\text{At}(L)$ to denote the set of all atoms of $L$ and write $X_{\text{iso}}$ to denote the set of all isolated points of $X$. 
\end{notation}
\begin{proposition}\label{atoms-iso points}
Given an ortholattice $L$ and its dual UVO-space $X$, the mapping $\text{At}(L)\to X_{\text{iso}}$ defined by $a\mapsto\uparrow\hspace{-.1cm}a$ is a bijection.  
\end{proposition}
\begin{proof}
Note that if $a\in\text{At}(L)$, then $\widehat{a}=\{\uparrow\hspace{-.1cm}a\}$ and since $\widehat{a}\in\mathcal{O}(X_L^+)$, it follows that $\uparrow\hspace{-.1cm}a$ is an isolated point. It immediately follows that the map is injective since clearly for all $a,b\in L$, if $a\not=b$ then without loss of generality, there exists some $c\in\uparrow\hspace{-.1cm}a$ such that $c\not\in\uparrow\hspace{-.1cm}b$ so $\uparrow\hspace{-.1cm}a\not=\uparrow\hspace{-.1cm}b$. 

To see that the map is a surjection, note that if $x$ is an isolated point, then $\{x\}$ is an open set and since $\mathcal{COR}(X)$ forms a basis for a UVO-space $X$, we have that $\{x\}\in\mathcal{COR}(X_L^+)$. Then by Theorem \ref{representation theorem}, there exists some $a\in L$ such that $\widehat{a}=\{x\}$ which implies that $a\in\text{At}(L)$. On the other hand, if $a\not\in\text{At}(L)$, then there exists some $0\not=b\in L$ such that $b<a$ but this implies that $\uparrow\hspace{-.1cm}a,\uparrow\hspace{-.1cm}b\in\mathfrak{F}(L)$ are such that $\uparrow\hspace{-.1cm}a\not=\uparrow\hspace{-.1cm}b$ with $\uparrow\hspace{-.1cm}a,\uparrow\hspace{-.1cm}b\in\widehat{a}$. Lastly, note that since $a\in\text{At}(L)$, we have $\widehat{a}=\{\uparrow\hspace{-.1cm}a\}$ which means that $x=\uparrow\hspace{-.1cm}a$.  
\end{proof}

\subsection{Atomic lattices and atomless lattices}
 Recall that a lattice $L$ is \emph{atomless} if $L$ contains no atoms and is \emph{atomic} if every element $a\in L$ can be written as a possibly infinite join of atoms. The following UVO-space characterization of an atomless ortholattice is an immediate corollary of Proposition \ref{atoms-iso points}.   
\begin{corollary}\label{atomless lattice} 
Let $L$ be an otholattice and let $X$ be its dual UVO-space. Then, $L$ is atomless if and only if $X_{\text{iso}}=\emptyset$. 
\end{corollary}
\begin{proof}
Since by Proposition \ref{atoms-iso points}, the collection of atoms \emph{At}(L) of an ortholattice $L$ are in bijection with the isolated points $X_{\emph{iso}}$ of its corresponding dual UVO-space $X$, it is clear that the collection of isolated points in $X$ is empty if and only if there exist no atoms in $L$.  
\end{proof}
\begin{proposition}
Let $L$ be an ortholattice and let $X$ be its dual UVO-space. Then, the following statements are equivalent:
\begin{enumerate}
    \item $L$ is atomic;
    \item $\text{Cl}(X_{\text{iso}})=X$.
\end{enumerate}
\end{proposition}
\begin{proof}
To show the forward implication,
assume that $L$ is complete,
and take $u \in X$.
It suffices to show that for every basic open neighborhood $U \ni u$,
the subset $X_{\text{iso}}$ intersects with $U$ nontrivially.
Find $a \in L \setminus \{0\}$ such that $U = \widehat a$.
By atomicity, there is an atom $b \le a$,
i.e., $\mathop\uparrow b \in \widehat a$, 
and $\mathop\uparrow b \in X_{\text{iso}}$.

To show the other implication,
let $a \in L \setminus \{0\}$ be arbitrary.
We will find an atom $b \le a$.
Consider $\mathop\uparrow a \in X$ and a neighborhood $\widehat a$.
Since $X_{\text{iso}}$ is dense, $\widehat a$ intersects
nontrivially with $X_{\text{iso}}$ at, say, $u$.
Recall that $u$ is of the form $\mathop\uparrow b$ for some atom $b$.
Since $\mathop\uparrow b \in \widehat a$, we have $b \le a$.
\end{proof}
\subsection{Injective and surjective homomorphism} 
\begin{definition}\label{uvo embedding}
Let $X$ and $Y$ be UVO-spaces. A UVO-map $f\colon X\to Y$ is a \textit{UVO-embedding} if $f$ is injective and for every $U\in\mathcal{COR}(X)$, there exists some $V\in\mathcal{COR}(Y)$ such that $f[U]=f[X]\cap V$. 
\end{definition}
\begin{proposition}\label{surj and embedding}
Let $L$ and $L'$ be ortholattices, let $h\colon L\to L'$ be an ortholattice homomorphism, and let $h_+\colon X^+_{L'}\to X^+_L$ be the corresponding dual UVO-map of $h$. 
Then, $h_+$ is a surjective UVO-map if $h$ is an injective ortholattice homomorphism.
Moreover, if $h_+$ is a UVO-embedding if $h$ is a surjective ortholattice homomorphism.
\end{proposition}
\begin{proof}
For the first part, assume that $h\colon L\to L'$ is an injective ortholattice homomorphism, and let $x\in X^+_L$ and $y=\{b\in L\mid \exists a\in h[x]:a\leq b\}$.  Clearly, $y$ is the inverse $h$-image of $x$.  We want to show that $y$ is a proper filter. 
To see this, note that if $0'\in y$, then $0'\in h[x]$ which implies the existence of some $a\in x$ such that $h(a)=0'$. By hypothesis, $x$ is a proper filter, which implies that $a\not=0$, but this contradicts the fact that $h(0)=0'$ together with our hypothesis that $h$ is injective. 
Hence,
it follows that $h_+$ is a surjective UVO-map.   

For the second part, let $x,y\subseteq L$ be filters such that $x\not=y$. Without loss of generality, there exists some $a\in L$ such that $a\in x$ but $a\not\in y$. By hypothesis, $h$ is a surjective ortholattice homomorphism and, therefore, there exists some $b\in L$ such that $h(b)=a$. 
It is easy to see that $b\in h^{-1}[x]$ and $b\not\in h^{-1}[y]$, which implies that $h^{-1}[x]\not= h^{-1}[y]$. Hence, $h_+$ is an injective UVO-map.

To see that $h_+$ satisfies the UVO-embedding condition,
first take an arbitrary $U\in\mathcal{COR}(X^+_{L'})$.
Note that by Theorem \ref{representation theorem},
$U$ is of the form $\widehat{a}$ for some $a\in L'$.
Again, by our hypothesis that $h$ is a surjective homomorphism,
there exists some $b\in L$ such that $h(b)=a$.
It suffices to show that $h_+[\widehat a] = h_+[X^+_{L'}] \cap \widehat b$.
To show the left-to-right inclusion,
take an arbitrary $x' \in h_+[\widehat a]$, i.e.,
$x = h^{-1}[x']$ for some $x \in \widehat a$.
We now have $b \in x$, and we are done.
To see the other inclusion,
take an arbitrary $x \in h_+[X^+_{L'}] \cap \widehat b$.
There exists $x' \in X^+_{L'}$ for which $x = h_+(x')$, i.e., $x = h^{-1}[x']$.
Since $b \in x$ as well, we conclude that $x' \in \widehat a$.
This shows that $x \in h_+[\widehat a]$.
\end{proof}
\begin{proposition}\label{inj and surj}
Let $X$ and $X'$ be UVO-spaces, let $f\colon X\to X'$ be a UVO-map, and let $f^+\colon\mathcal{COR}(X')\to\mathcal{COR}(X)$ be the corresponding  ortholattice homomorphism dual to $f$. Then, $f^+$ is an injective ortholattice homomorphism if $f$ is a surjective UVO-map.
Moreover, the map $f^+$ is a surjective ortholattice homomorphism if $f$ is a UVO-embedding. 
\end{proposition}
\begin{proof}
For the first part, let $X$ and $X'$ be UVO-spaces, and let $f\colon X\to X'$ be a surjective UVO-map. Now suppose that $U,V\in\mathcal{COR}(X)$ are such that $U\not=V$. Without loss of generality, if $y\in U\setminus V$, then since $f$ is surjective, there exists some $x\in X$ such that $f(x)=y$ so $x\in f^{-1}[U]$ and $x\not\in f^{-1}[V]$. Since $f^{-1}[U]=f^+(U)$ and $f^{-1}[V]=f^+(V)$, we have $f^+(U)\not=f^+(V)$. Hence, $f^+$ is an injective ortholattice homomorphism. 

For the second part, let $X$ and $X'$ be UVO-spaces and let $f\colon X\to X'$ be a UVO-embedding. If $U\in\mathcal{COR}(X)$, then since $f$ is a UVO-embedding, by Definition \ref{uvo embedding}, there exists some $V\in\mathcal{COR}(X')$ such that $f[U]=f[X]\cap V$, which implies that $f^{-1}[f[U]]=f^{-1}[f[X]\cap V]$. Now observe that \[f^{-1}[f[X]\cap V]=f^{-1}[f[X]]\cap f^{-1}[V]=X\cap f^{-1}[V]=f^{-1}[V].\] By hypothesis, $f$ is a UVO-embedding and therefore injective, which guarantees that $f^{-1}[f[U]]=U$ so $f^{-1}[V]=U$ and since $f^{-1}[V]=f^+(V)$, we have $f^+(V)=U$, as desired.  
\end{proof}
\subsection{Subalgebra}
\begin{corollary}
Let $L$ be an ortholattice and let $X$ be its dual UVO-space. Then, there exists a one-to-one correspondence between the subalgebras of $L$ and the images via surjective UVO-maps of $X$. 
\end{corollary}
\begin{proof}
The result follows immediately by Theorem \ref{main theorem}, the first part of Proposition \ref{surj and embedding} (i.e., that $h_+$ is a surjective UVO-map if its dual ortholattice homomorphism $h$ is injective), and the first part of Proposition \ref{inj and surj} (i.e., that $f^+$ is an injective ortholattice homomorphism if its dual UVO-map $f$ is surjective).  
\end{proof}
\subsection{Direct product}
\begin{definition}\label{uvo sum}
  If $X$ and $Y$ are UVO-spaces, then their \emph{UVO-sum} $X+Y$ is the space whose underlying carrier set is of the following shape \[X+Y:=X\cup Y\cup (X\times Y)\] and whose topology is generated by sets of the form $U\cup V\cup(U\times V)$ for $U\in\mathcal{COR}(X)$ and $V\in\mathcal{COR}(Y)$,
  together with the orthogonality relation $\perp_{X+Y}$,
  which is defined as the symmetric closure of:
  \begin{align*}
    \bot_X\cup\bot_Y &\cup (X \times Y) \\
    &\cup\{\langle\langle x,y\rangle,x'\rangle\mid x\perp_Xx'\}\cup\{\langle\langle x,y\rangle,y'\rangle\mid y\perp_Yy'\}\\
    &\cup
    \{\langle x,y\rangle,\langle x',y'\rangle\mid x\perp_X x',y\perp_Yy'\}.
  \end{align*}
\end{definition}

\begin{proposition}
Let $X$ and $Y$ be UVO-spaces whose specialization orders are $\leqslant_X$ and $\leqslant_Y$ respectively. Then, the specialization order $\leqslant_{X+Y}$ of their UVO-sum $X+Y$ is given by:  \[\Omega_{\leqslant}:=\leqslant_X\cup\leqslant_Y\cup\{\langle\langle x,y\rangle,x'\rangle\mid x\leqslant_Xx'\}\cup\{\langle\langle x,y\rangle,y'\rangle\mid y\leqslant_Yy'\}\cup\]\[\{\langle x,y\rangle,\langle x',y'\rangle\mid x\leqslant_X x',y\leqslant_Yy'\}.\]
\end{proposition}
\begin{proof}
  Assume that $\langle z,z'\rangle\in\Omega_{\leqslant}$ such that $z\in W= U\cup V\cup (U\times V)\in\mathcal{O}(X+Y)$ for $U\in\mathcal{COR}(X)$ and $V\in\mathcal{COR}(Y)$. We want to show that $z'\in W$. In the case when $z\leqslant_Xz'$, we have $z\in U\in\mathcal{COR}(X)$ so $z'\in U\in\mathcal{COR}(X)$, hence $z'\in W$. In the case when $z=\langle x,y\rangle$, we have $\langle x,y\rangle\in U\times V$ with $x\in U\in\mathcal{COR}(X)$ and $y\in V\in\mathcal{COR}(Y)$. Thus, if $x\leqslant_Xz'$, it follows that $z'\in U\in\mathcal{COR}(X)$ and therefore, $z'\in W$. The proof of the case for $z\leqslant_Yz'$ and the case for $z'=\langle x',y'\rangle$, $x\leqslant_Xx'$, and $y\leqslant_Yy'$ run analogously, as does the converse direction under the assumption that $\langle z,z'\rangle\not\in\Omega_{\leqslant}$.  
\end{proof}
The following result follows from Theorem \ref{characterization theorem} and is useful in diagrammatically presenting examples of finite UVO-spaces in terms of the specialization ordering of their points.   
\begin{proposition} 
  If $X$ is a finite UVO-space, then $(X,\leqslant)$ can be constructed from an ortholattice by deleting its bottom universal bound
  and taking its order-theoretic dual.  
\end{proposition}
\begin{proof}
  First note that by Lemma \ref{space to lattice}, if $X$ is a finite UVO-space, then $\mathcal{COR}(X)$ is a finite ortholattice. By Theorem \ref{characterization theorem}, $X$ is homeomorphic to the space $X^+_{\mathcal{COR}(X)}$ of proper lattice filters of $\mathcal{COR}(X)$, which is a $T_0$ space by Lemma \ref{lattice to uvo-space}. Hence, $(X,\leqslant)$ is order isomorphic to the poset $(X^+_{\mathcal{COR}(X)},\subseteq)$ of proper lattice filters of $\mathcal{COR}(X)$ ordered by set-theoretic inclusion. Since any filter of a finite ortholattice is a principal filter, we have that $(X,\leqslant)$ is isomorphic with respect to the ortholattice of proper principal filters of $\mathcal{COR}(X)$, which is isomorphic  to the lattice $(\mathcal{COR}(X)\setminus \{\emptyset\}, \supseteq)$.   
\end{proof}

\begin{example}
Consider the ortholattices $O_2$ and $O_6$, along with their respective UVO-spaces $X^+_{O_2}$ and $X^+_{O_6}$ depicted in Figure \ref{examples of uvo spaces}. The direct product $O_2\times O_6$ and its UVO-sum $X^+_{O_2}+X^+_{M_3}$ are depicted in Figure \ref{example of uvo sum}. For the purpose of visualizing these examples, we represent the specialization order between two points of a UVO-space by a dotted line and the underlying partial ordering of an ortholattice by a solid line. We omit the orthogonality relations of these UVO-spaces from these diagrams for the sake of simplicity, but still explicitly describe them below.  
\[\perp_{X^+_{O_6}}=\{\langle y_1,y_2\rangle,\langle y_1,y_4\rangle,\langle y_3,y_2\rangle,\langle y_3,y_4\rangle\},\hspace{.5cm}\perp_{X^+_{O_2}}=\emptyset\]
\[\perp_{X^+_{O_2}+X^+_{O_6}}=\{\perp_{X^+_{O_6}},\perp_{X^+_{O_2}},\langle\langle z,y_1\rangle,y_2\rangle,\langle\langle z,y_1\rangle,y_4\rangle,\langle\langle z,y_2,\rangle y_1\rangle,\]\[\langle\langle z,y_2\rangle,y_3\rangle,\langle\langle z,y_3\rangle,y_2\rangle,\langle\langle z,y_3\rangle,y_4\rangle,\langle\langle z,y_4\rangle,y_1\rangle,\langle\langle z,y_4\rangle,y_3\rangle\}\]
\end{example}
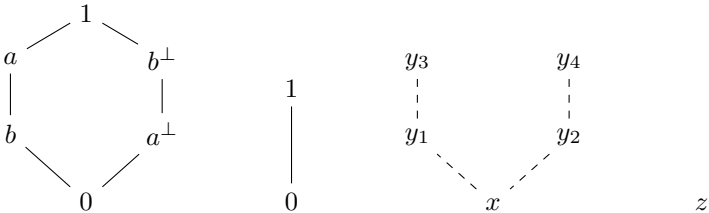
\begin{figure}[htbp]
  \begin{tikzpicture}
  \node (a) at (0,2.25) {$1$};
  \node (d) at (0,-.25) {$0$};
  \node (e) at (-1,.65) {$b$};
  \node (f) at (-1,1.65) {$a$};
    \node (g) at (1,.65) {$a^{\perp}$};
  \node (h) at (1,1.65) {$b^{\perp}$};
  \draw (e) -- (f) (g) -- (h) (e) -- (d) (g) -- (d) (f) -- (a) (h) -- (a) (a);
  \draw[preaction={draw=white, -,line width=6pt}];
\end{tikzpicture}
\hskip 3em
\begin{tikzpicture}
  \node (a) at (0,1.25) {$1$};
  \node (d) at (0,-0.25) {$0$};
  \draw (a) -- (d);
  \draw[preaction={draw=white, -,line width=6pt}];
\end{tikzpicture}
\hskip 3em
\begin{tikzpicture}
\node (e) at (1,1.65) {$y_4$};
\node (a) at (-1,1.65) {$y_3$};
  \node (b) at (1,.65) {$y_2$};
  \node (c) at (-1,.65) {$y_1$};
  \node (d) at (0,-.25) {$x$};
  \draw[dashed] (b) -- (d) (d) -- (c) (a) -- (c) (e) -- (b);
  \draw[preaction={draw=white, -,line width=6pt}];
\end{tikzpicture}    
\hskip 3.25em
\begin{tikzpicture}
  \node (d) at (0,.5) {$z$};
\end{tikzpicture}
\caption{The ortholattice $O_6$, the ortholattice algebra $O_2$, the UVO-space $X^+_{O_{6}}$, and the UVO-space $X^+_{O_{2}}$}\label{examples of uvo spaces}
\end{figure}
\begin{figure}[htbp]
  \begin{tikzpicture}
  \node (a) at (3.75,3.25) {$\langle 0,1\rangle$};
    \node (b) at (4.5,1.5) {$\langle0,b^{\perp}\rangle$};
  \node (c) at (1.5,1.5) {$\langle 1,b\rangle$};
  \node (d) at (2.5,-0.5) {$\langle 0,b\rangle$};
  \node (e) at (3,1.5) {$\langle0,a\rangle$};
  \node (g) at (3.75,-0.5) {$\langle 1,0\rangle$};
  \node (h) at (5, -0.5) {$\langle 0,a^{\perp}\rangle$}; 
  \node (i) at (6,1.5) {$\langle1,a^{\perp}\rangle$};
  \node (j) at (3.75,-2.5) {$\langle 0,0\rangle$};
  \node (k) at (1.5,3.25) {$\langle 1,a\rangle$};
  \node (l) at (6,3.25) {$\langle1,b^{\perp}\rangle$};
  \node (m) at (3.75,4.5) {$\langle1,1\rangle$};
  \draw (d) -- (c) (d) -- (e) (b) -- (h) (j) -- (g) (j) -- (h) (j) -- (d) (i) -- (g) (i) -- (h) (e) -- (a) (g) -- (c) (b) -- (a) (c) -- (k) (i) -- (l) (a) -- (m) (k) -- (m) (l) -- (m) (k) -- (e) (b) -- (l);
  \draw[preaction={draw=white, -,line width=6pt}];
  \end{tikzpicture}
\hskip .5em
  \begin{tikzpicture}
  \node (a) at (3.75,3.25) {$z$};
    \node (b) at (4.5,1.5) {$y_2$};
  \node (c) at (1.5,1.5) {$\langle z,y_3\rangle$};
  \node (d) at (1.5,-0.5) {$\langle z,y_1\rangle$};
  \node (e) at (3,1.5) {$y_1$};
  \node (g) at (3.75,-0.5) {$x$};
  \node (h) at (6, -0.5) {$\langle z,y_2\rangle$}; 
  \node (i) at (6,1.5) {$\langle z,y_4\rangle$};
  \node (j) at (3.75,-2.5) {$\langle z,x\rangle$};
  \node (k) at (2.75,3.25) {$y_3$};
  \node (l) at (5,3.25) {$y_4$};
  \draw[dashed] (j) -- (d) (j) -- (g) (j) -- (h) (g) -- (e) (g) -- (b) (d) -- (c) (d) -- (e) (h) -- (b) (h) -- (i) (a) -- (c) (a) -- (i) (k) -- (c) (k) -- (e) (l) -- (b) (l) -- (i);
  \draw[preaction={draw=white, -,line width=6pt}];
  \end{tikzpicture}\caption{The direct product ortholattice $O_2\times O_6$ and its dual UVO-sum $X^+_{O_{2}}+X^+_{O_6}$}\label{example of uvo sum}
\end{figure}
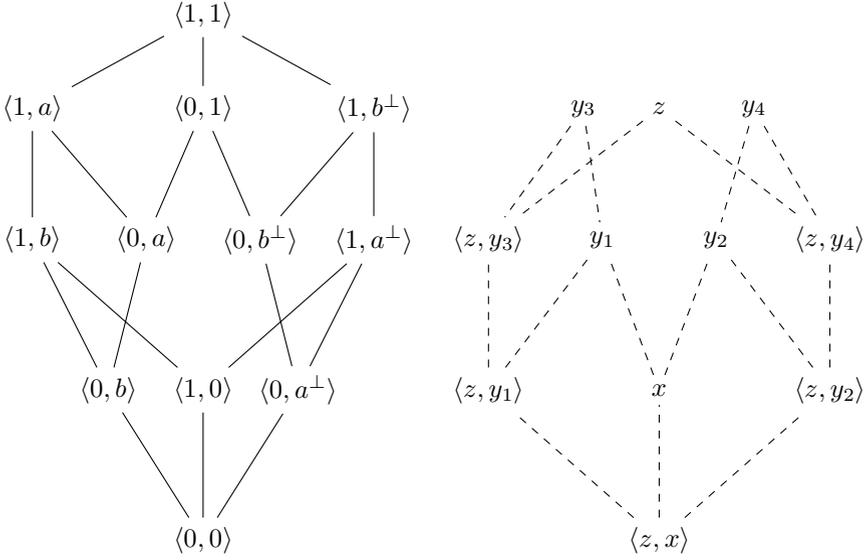

\begin{proposition}\label{homeo}
If $L$ and $L'$ are ortholattices and $X^+_L$ and $X^+_{L'}$ are their respective dual UVO-spaces, then there is a homeomorphism $f\colon X^+_{L\times L'} \to X^+_L+X^+_{L'}$ that is an isomorphism with respect to their orthospace reducts.  
\end{proposition}
\begin{proof}
  For each point $x\in X^+_{L\times L'}$ i.e., every proper filter $x\in \mathfrak{F}(L\times L')$, let
  \[x_L=\{a\in L\mid \exists b\in L':\langle a,b\rangle\in x\},\hspace{.2cm} x_{L'}=\{b\in L'\mid \exists a\in L:\langle a,b\rangle\in x \},\] Clearly, we have that $x_L\in\mathfrak{F}(L)$ and $x_{L'}\in\mathfrak{F}(L')$. Now define $f$ by  
\[
    f(x)=
\begin{cases}
    x_{L}& \text{if  $x_{L'}$ is improper,}\\
    x_{L'} &\text{if $x_{L}$ is improper,} 
    \\
    \langle x_L, x_{L'}\rangle              & \text{otherwise.}     
\end{cases}
\] 

The injectivity of $f$ follows easily from the fact
that $x = x_L \times x_{L'}$ for every filter $x \in X^+_{L \times L'}$.
To see that $f$ is a surjective function, let $y\in X^+_{L}+X^+_{L'}$. 
In the case when $y\in X^+_L$, we have that for the proper filter $x := y \times L' \in X^+_{L \times L'}$, it follows that $y=x_{L'}$ and that $x_L$ is improper. Therefore, we find that $f(x)=y$. The proof for the case when $y\in X^+_{L'}$ runs analogously. Lastly, for $y^L\in X^+_L$ and $y^{L'}\in X^+_{L'}$, in the case when $y=\langle y^{L},y^{L'}\rangle$, since $(y^L\times y^{L'})_L=y^L$ and $(y^L\times y^{L'})_{L'}=y^{L'}$, it is easy to see that $y^L\times y^{L'}\in X^+_{L\times L'}$ where $f(y^L\times y^{L'})=y$. Hence, $f$ is a bijection.    

We now verify that $f$ is a continuous function. First observe that by Definition \ref{uvo sum}, each basic open set within $X^+_L+X^+_{L'}$ is of the following shape $U\cup V\cup(U\times V)$ for $U\in\mathcal{COR}(X^+_L)$ and $V\in\mathcal{COR}(X^+_{L'})$. By Theorem \ref{representation theorem}, each $U\in\mathcal{COR}(X^+_L)$ is of the form $\widehat{a}$ for some $a\in L$, and so \[U\cup V\cup (U\times V)=\widehat{a}\cup\widehat{b}\cup(\widehat{a}\times\widehat{b})\] for $a\in L$ and $b\in L'$. We now verify that the inverse image of each basic open set is a union of basic open sets in $X^+_{L\times L'}$ by the following calculation: \[f^{-1}[\widehat{a}\cup\widehat{b}\cup(\widehat{a}\times\widehat{b})]=f^{-1}[\widehat{a}]\cup f^{-1}[\widehat{b}]\cup f^{-1}[\widehat{a}\times\widehat{b}]=\widehat{\langle a,0\rangle}\cup\widehat{\langle0,b\rangle}\cup\widehat{\langle a,b\rangle}\] Hence, $f^{-1}[\widehat{a}\cup\widehat{b}\cup(\widehat{a}\times\widehat{b})]$ can be written as the union of basic open sets in the space $X^+_{L\times L'}$, so $f$ is a continuous function. To see that its inverse $f^{-1}$ is a continuous function, note that for each basic open set $\langle a,b\rangle\in X^+_{L\times L}$,    
\begin{align*}
    \widehat{\langle a,b\rangle}&=\{x\in\mathfrak{F}(L\times L')\mid\langle a,b\rangle\in x, x_{L'}\text{ is improper}\}\\&\cup\{x\in\mathfrak{F}(L\times L')\mid\langle a,b\rangle\in x, x_{L}\text{ is improper}\}\\&\cup\{x\in\mathfrak{F}(L\times L')\mid\langle a,b\rangle\in x:x_{L}\in\mathfrak{F}(L),x_{L'}\in\mathfrak{F}(L')\}
\end{align*}
which implies that $f[\widehat{\langle a,b\rangle}]=\widehat{a}\cup\widehat{b}\cup(\widehat{a}\times \widehat{b})$ so that $f[\widehat{\langle a,b\rangle}]$ is basic open in the space $X^+_L+X^+_{L'}$, as required.

Finally, we show that $f$ is an isomorphism with respect
to the orthospace reducts.
Let $\perp_{\text s}$ and $\perp$ be the orthogonality relations
of the codomain and the domain of $f$,  respectively.
The preceding argument shows that the inverse map $f^{-1}$ of $f$
is given by $f^{-1}(x) = x \times L'$, $ f^{-1}(y) = L \times y$,
and $f^{-1}(x, y) = x \times y$, where $x \in X^+_L$ and $y \in X^+_{L'}$.
Let $u, v \in X^+_L + X^+_{L'}$.
An argument showing that $u \perp_{\text s} v$ if and only if
$f^{-1}(u) \perp f^{-1}(v)$
involves a case analysis based on whether
$u$ and $v$ belong to $X^+_L$, $X^+_{L'}$, or $X^+_L \times X^+_{L'}$.
We present an argument for the case $u \in X^+_L $
and $v = \langle w, w'\rangle  \in X^+_L \times X^+_{L'}$
as the other cases can be handled in similar ways.
By the definition of $\perp_{\text s}$,
we have that $u \perp_{\text s} v$ if and only if 
there exists $a \in w$ such that $a^\bot \in u$.
On the other hand, 
\begin{align*}
  \inv f u \perp \inv f v &\iff u \times L' \perp w \times w'\\
  &\iff \exists \langle a, a'\rangle \in w \times w': \langle a^\bot, a'^\bot\rangle \in u \times L'\\
  &\iff \exists a \in w: a^\bot \in L,
\end{align*}
proving the claim for this particular case.
\end{proof}
\begin{corollary}
If $X$ and $Y$ are UVO-spaces, then their UVO-sum $X+Y$ is a UVO-space. Moreover, the mapping $f\colon\mathcal{COR}(X+Y)\to\mathcal{COR}(X)\times\mathcal{COR}(Y)$ is an ortholattice isomorphism. 
\end{corollary}
\begin{proof}
Clearly, by Theorem \ref{characterization theorem}, the maps $X\to X^+_{\mathcal{COR}(X)}$ and $Y\to X^+_{\mathcal{COR}(Y)}$ are homeomorphisms (and isomorphisms with respect to $\perp$) and thus there is a homeomorphism $X+Y\to X^+_{\mathcal{COR}(X)}+X^+_{\mathcal{COR}(Y)}$. Then by Proposition \ref{homeo}, we find that $X^+_{\mathcal{COR}(X)}+X^+_{\mathcal{COR}(Y)}\to X^+_{\mathcal{COR}(X)\times\mathcal{COR}(Y)}$ is a homeomorphism. The above homeomorphisms are sufficient in establishing the fact that the UVO-sum $X+Y$ is a UVO-space if $X$ and $Y$ are UVO-spaces. For the second part, simply apply Theorem \ref{main theorem} and Proposition \ref{homeo}. 
\end{proof}
It is easy to check that every UVO-sum $X+Y$ comes equipped with canonical coprojections $\kappa_1\colon X\to X+Y$ and $\kappa_2\colon Y\to X+Y$ satisfying the universal mapping property for categorical coproducts that for any UVO-space $Z$ and pair of UVO-maps $f\colon X\to Z$ and $g:Y\to Z$, there exists a unique UVO-map $\langle f,g\rangle\colon X+Y\to Z$ making the following diagram commute:
  \[
    \begin{tikzcd}[row sep=huge]
        & Z  & \\
        X\ar[ur,"f",sloped] \ar[r,"\kappa_1", swap] & X+Y \ar[u,dashed,"{\langle f,g\rangle}" description] & Y \ar[ul,"g",sloped] \ar[l,"\kappa_2"]
    \end{tikzcd}
    \]
Hence, given any two UVO-spaces $X$ and $Y$, their UVO-sum $X+Y$ is a coproduct in the category $\mathbf{UVO}$. As a consequence, it follows that if $X_1,\dots,X_n$ are UVO-spaces, then $X_1 + \dots + X_n$ is a UVO-space. 

\subsection{Lattice completions}
\begin{notation}
Let $L$ be a lattice and let $A\subseteq L$. Then:
\begin{enumerate}
    \item $A^u$ is the collection of upper bounds of $A$, i.e., \[A^u=\{a\in L\mid\forall b\in A: b\leq a\}.\]
    \item $A^l$ is the collection of lower bounds of $A$, i.e., \[A^l=\{a\in L\mid\forall b\in A:a\leq b\}.\]
\end{enumerate}
\end{notation}
\begin{definition}
Given a lattice $L$, a subset $A\subseteq L$ is \textit{normal} iff $A=A^{\textit{ul}}$. We denote the collection of all normal subsets of $L$ by $\text{Norm}(L)$. 
\end{definition}

We call a point $u$ of a UVO-space $X$ \emph{principal}
if there exists an open neighborhood $U$ of $u$
such that $v \not \in U$ for every $v \leqslant u$ distinct from $u$.
\begin{proposition}
  Let $L$ be an  ortholattice and $X$ its dual UVO-space.
  A point in $X$ is principal in the sense above
  if and only if it is a principal filter.
\end{proposition}
\begin{proof}
  It is clear that if $u \in X$ is a principal filter,
  then it is principal in the sense above.
  Suppose that $u \in X$ is principal in our sense.
  Take a neighborhood $U$ of $u$ as in the definition of principality
  and then a basic open set $\widehat a$ such that $u \in \widehat a \subseteq U$.
  Let $v$ be the principal filter generated by $a$.
  Assume by way of contradiction that $u$ is not a principal filter.
  Then, $v \leqslant u$ and $v \neq u$.
  By principality, we have $v \not \in U$
  and \emph{a fortiori} $v \not \in \widehat a$, which is a contradiction.
\end{proof}
For a UVO-space $X$,
let $\mathfrak{P}(X)$ be the orthoframe of principal points of $X$
with the induced orthogonality relation.
We then have the following UVO-space translation of the MacNeille completion of an ortholattice. 
 \begin{theorem}
   Let $L$ be an ortholattice and let $X$ be its dual UVO-space. Then, the lattice $\mathcal{R}(\mathfrak{P}(X))$ is (up to isomorphism) the MacNeille completion of $L$.
   
 \end{theorem}
 \begin{proof}
   \newcommand{\ppl}{\mathop{\uparrow}}
   MacLaren \cite[Theorems 2.3--2.5]{maclaren} showed that the MacNeille completion of $L$
   is isomorphic to $\mathcal{R}(L, \mathord{\simperp})$,
   where $\simperp$ is a binary relation on (the domain of) $L$
   defined by $a \simperp b \iff a \le b^\bot$.
   We see that $\mathcal R(L, \mathord\simperp)$ is isomorphic to
   $\mathcal R(L^-, \mathord\simperp)$, where
   $(L^-, \mathord\simperp)$ is the relational substructure of $L$
   with the domain $L^- = L \setminus \{0\}$.
   To see this, first observe that $0 \simperp a$ for all $a \in L$
   and that $0 \in U$ for all $U\in\mathcal R(L, \mathord\simperp)$.
   From this, it follows that $U \mapsto U \setminus \{0\}$
   for $U \in \mathcal R(L, \mathord\simperp)$
   is the desired isomorphism.
   It now suffices to show that $(L^-, \mathord\simperp)$ is isomorphic to
   $\mathfrak P(X) = (\mathfrak P(X), \bot)$.
   To see this,
   first note that for an arbitrary $c \in L$ and $u \in X$,
   we have $u \perp \ppl c$,
   where $\ppl c$ is the principal filter generated by $c$,
   if and only if $c^\bot \in u$.
   Hence, $(\ppl a) \perp (\ppl b)$ if and only if $b^\bot \in \ppl a$,
   i.e., $b^\bot \ge a$.
\end{proof}

We proceed with a characterization of canonical extensions as defined by Harding~\cite[p.~92]{tatra} (see also \cite[Theorem 8.7]{bruns}).
\begin{theorem}
Let $L$ be an ortholattice and let $X$ be its dual UVO-space. Then $L' := \mathcal{R}(X)$ is (up to isomorphism) the canonical extension of $L$. 
\end{theorem}
\begin{proof}
\renewcommand{\bot}{*}
For $u \in X^+_L$, the set $\{u\}^{\bot\bot} \in L'$ is a meet of elements of $L$:
$u = \bigwedge \{ \widehat a \mid \widehat a \supseteq \{u\}^{\bot\bot} \}$.
The inclusion $\subseteq$ is clear.
To show $\supseteq$,
take $v \not \in \{u\}^{\bot\bot}$.
If $u \le v$, then $\{u\}^\bot \subseteq \{v\}^\bot$,
so $v \in \{u\}^{\bot\bot}$;
hence, $u \not \le v$.
Take $a \in u \setminus v$;
then $u$ is in $\widehat a$, but $v$ is not
($\widehat\bullet$ denotes the embedding $L \to L'$).
Note that $u =
\bigcap \{ \widehat a \mid \widehat a \supseteq \{u\}^{\bot\bot} \}$.
We have seen that $\{u\}^{\bot\bot} \in L'$ is a meet of elements of $L$.
We now show that every element of $L'$ is a meet of joins of elements of $L$. 
In particular, we claim that for $Y \in L'$ we have
$Y = \bigvee \{\{u\}^\bot \mid Y \supseteq \{u\}^\bot\}$.
The inclusion $\subseteq$ is clear.
To show the inclusion in the other direction,
we show the contrapositive:
if $\{u\}^\bot \supseteq Y \implies \{u\}^\bot \ni v$ for every $u$,
then $v \in Y$.
Assume the hypothesis;
we show $v \perp Y^\bot$.
Take an arbitrary $u \in Y^\bot$.
Then $\{u\}^\bot \supseteq Y$, so we have $\{u\}^\bot \ni v$,
i.e., $u \perp v$.

Lastly, we verify that the embedding $L \to L'$ is compact.
Assume that
\begin{equation}
\bigwedge_i \widehat{a_i} \subseteq \bigvee_j \widehat{b_j}\label{eq:infinitary}
\end{equation}
for families $A := \{a_i \mid i \in I\}, B := \{b_j \mid j \in J\}$
of elements of $L$.
We show that there exists finite subfamilies $A', B'$ of $A, B$, respectively,
such that $\bigwedge A \le \bigvee B$.
Note that the right-hand side of formula~(\ref{eq:infinitary}) is equal to
$(\bigcup_j\widehat{b_j})^{\bot\bot}$.
Now let $u$ be the filter generated by $A$.
Assume that $u$ is improper.
Then there exists a finite $A' \subseteq A$
such that $0 = \bigwedge A'$.
Now we let $B' = \emptyset$, and we are done.
Suppose, therefore, that $u$ is proper.
By construction, $u$ is in the left-hand side of formula~(\ref{eq:infinitary})
and thus in the right-hand side.
Observe
\begin{equation}\label{eq:botbot}
  u \in \left(\bigcup_j\widehat{b_j}\right)^{\bot\bot}
  \iff
  \zenbu{v}[\zenbu{w}[\aru{j}b_j \in w \implies w \perp v] \implies u \perp v],
\end{equation}
where the variables $v$ and $w$ range over $X$.
Let $v$ be the filter generated by $\{b_j^\bot \mid j \in J\}$.
Assume that $v$ is improper.
Then for some finite $B' \subseteq B$,
we have $0 = \bigwedge \{b^\bot \mid b \in B'\}$,
i.e., $1 = \bigvee B'$.
Therefore, by a similar reasoning as before,
we may assume that $v$ is proper.
By formula~(\ref{eq:botbot}), $u \perp v$.
\renewcommand{\bot}{{\mathord\perp}}
By definition, there exists  $c \in u$ such that $c^\bot \in v$.
By the construction of $u$ and $v$,
we have $\bigwedge A' \le c$ and $c^\bot \ge \bigwedge \{b^\bot \mid b \in B'\}$.
The latter implies $c \le \bigvee B'$,
so we have $\bigwedge A' \le \bigvee B'$ as desired.
\end{proof}
\subsection{Homomorphic images of orthomodular lattices}
We conclude this section by characterizing the notion of homomorphic images as applied to an orthomodular lattice, in UVO-spaces. We leave the characterization of homomorphic images as applied to ortholattices (the more general case) as an open problem.

Recall that a subset $S'$ of a relational structure $(S, R)$
where $R$ is binary
is an \emph{inner substructure}, or a \emph{generated subframe} $(S, R)$,
if $y \in S'$ whenever $x \in S'$ and $xRy$.
For the remainder of this subsection,
upsets \emph{simpliciter} mean
sets upward closed with respect to the specialization order $\leqslant$.
\begin{proposition}
  Let $L$ be an orthomodular lattice and $X$ be its dual UVO-space.
  Let $C(L)$ be the set of congruences on $L$
  and $\mathrm{PUGS}(L)$  the set of
  principal upsets of $X$ that are generated subframes of $(X, \not \perp)$.
  Then there is a one-to-one correspondence between
  $C(L)$ and $\mathrm{PUGS}(L)$.
\end{proposition}
\begin{proof}
  \newcommand{\ppl}{\mathop{\Uparrow}}
  For  $\theta \in C(L)$,
  it is well known that $[1]_\theta$ is a filter.
  Let $f(\theta) = \ppl [1]_\theta$,
  where $\ppl u$ for $u \in X$ is the principal upset generated by $u$.
  We see that $f(\theta) \in \mathrm{PUGS}(L)$
  and that $f$ is a map $C(L) \to \mathrm{PUGS}(L)$.
  Indeed, it suffices to show that $f(\theta)$ is a generated subframe
  with respect to the complement of the orthogonality relation of $X$.
  Consider the canonical surjection $\pi\colon L \twoheadrightarrow L/\theta$.
  The dual map $\pi^+$ is a UVO-map and \emph{a fortiori}
  a homeomorphism onto a subspace of $X$.
  We claim that $\ran \pi^+$, the range of $\pi^+$, is $f(\theta)$,
  whence it follows that $f(\theta)$ is a generated subframe
  as $\pi^+$ is weakly p-morphic and $f(\theta)$ is clearly upward closed.
  To see that $\ran \pi^+ = f(\theta)$,
  first recall that
  $u \in \ran \pi$
  if and only if there exists $u' \in \mathfrak F(L/\theta)$
  such that $ \pi^{-1}[u'] = u$.
  For every $u' \in \mathfrak F(L/\theta) $, we have $[1]_\theta \in u'$.
  Hence, if $u \in \ran \pi^+$, then $[1]_\theta \subseteq u$.
  Conversely, if $[1]_\theta \subseteq u$,
  assume $a \in u$ and $(a, a') \in \theta$ for $a, a' \in L$.
  We show that $a' \in u$, i.e., $u \in \ran \pi^+$.
  \newcommand{\shook}{\mathbin{\rightarrow}}
  Let $\shook$ be the so-called \emph{Sasaki hook},
  i.e., $x \shook y := x^\bot \vee (y \wedge x)$
  (see, e.g., \cite{pavicic}).
  We have $\pi(a \shook a') = \pi(a) \shook \pi(a') = 1$ by assumption.
  Therefore, $a \shook a' \in [1]_\theta \subseteq u$.
  Since $a \wedge (a\shook a') \in u$, we have $a' \in u$ as well.

  For $S \in \mathrm{PUGS}(X)$,
  let $g(S) = \{(a, b) \in L^2 \mid \widehat a \cap S = \widehat b \cap S\}$.
  We show that $g(S)$ is a congruence on $L$
  and that $g$ is a map $\mathrm{PUGS}(X) \to C(X)$.
  It suffices to show that $g(S)$ respects $\wedge$ and $(\bullet)^\bot$.
  The former case is evident.
  For the latter goal,
  it suffices to show that for $a, b \in L$ if
  $\widehat a \cap S = \widehat b \cap S$,
  then $\widehat{a^\bot} \cap S = \widehat{b^\bot} \cap S$.
  This can be proved by the translation into the normal modal logic KTB of reflexive and symmetric frames.

  It is easy to show that $f$ and $g$ are the inverses of each other
  by noting \[[1]_{g(\ppl u)} = \{a \in L \mid \widehat a \cap \ppl u = \ppl u\} = \{a \mid \widehat a \subseteq \ppl u\} = \{a \mid u \in \widehat a \}= u,\] which completes the proof. 
\end{proof}

\section{Future work}
 We intend to investigate the following applications and related themes of the results and constructions achieved in this work:
 
\begin{enumerate}
    \item Characterize the subclass of UVO-spaces which arise as the choice-free dual spaces of the modular and orthomodular lattices.  
    \item Develop a theory of topological models based on UVO-spaces for which various non-classical logics (e.g. orthologic and quantum logic) are complete. (Since the algebraic model for quantum logics of a finite dimensional Hilbert space is a modular lattice and the algebraic model for quantum logics of an infinite dimensional Hilbert space is an orthomodular lattice, the open problem of characterizing the dual UVO-spaces of the modular and orthomodular lattices must be accomplished before this can be fully addressed).   
    \item Investigate the connections between lattices of varieties of ortholattices and lattices of varieties of modal algebras corresponding to KTB (the normal modal logic of reflexive symmetric Kripke frames) and its variants. Such frames can be seen as arising by taking the set-theoretic complement of the orthogonality relation of an orthospace. We believe Goldblatt in \cite{goldblatt2} develops the first step in this direction.

\end{enumerate}

\section*{Acknowledgements} Both researchers thank their respective supervisors, Dr. Nick Bezhanishvili and Professor Wesley Holliday, for their expertise and guidance throughout the preparation of this paper. We also thank Dr. Tommaso Moraschini, as well as the participants at the Algebra/Coalgebra Seminar (Institute for Logic, Language, and Computation, University of Amsterdam), the BLAST 2021 conference (Department of Mathematical Sciences, New Mexico State University), and the Logica 2021 conference (Institute of Philosophy, Czech Academy of Sciences). Lastly, we thank the anonymous reviewer at \emph{Algebra Universalis} for their helpful comments and suggestions.

\end{document}